\documentclass[reqno]{amsart}
\usepackage{amssymb,latexsym}
\usepackage{amsmath}
\usepackage{amscd}
\usepackage{graphicx}
\usepackage{tikz}
\usepackage{color}
\usepackage{enumerate}
\newenvironment{enumeratei}{\begin{enumerate}[\upshape (i)]}{\end{enumerate}}
\numberwithin{equation}{section}

\theoremstyle{plain}
\newtheorem{theorem}{Theorem}[section]
\newtheorem*{theorem*}{Theorem}
\newtheorem{lemma}[theorem]{Lemma}
\newtheorem{corollary}[theorem]{Corollary}

\theoremstyle{definition}
\newtheorem{definition}[theorem]{Definition}
\newtheorem{remark}[theorem]{Remark}

\DeclareMathOperator{\Aut}{Aut}
\newcommand \DAstr{DAut}
\DeclareMathOperator{\Daut}{\DAstr}
\DeclareMathOperator{\Fl}{FL}
\newcommand \fl [1] {\Fl(#1)}
\DeclareMathOperator{\Sl}{SFL}
\DeclareMathOperator{\cfree}{CF}

\DeclareMathOperator{\Sym}{Sym}
\newcommand{\alg}[1]{\mathbf{#1}}

\newcommand {\jhom} {\nu^{(\vee)}}
\newcommand {\tno} {N_{\sst{\textup{var}}}}
\newcommand {\mhom} {\mu^{(\wedge)}}
\newcommand \tbf [1] {\textbf{#1}} 
\newcommand \set[1] {\{#1\}}
\newcommand \sst [1] {\scriptscriptstyle{#1}}
\newcommand \ovl [1] {\overline{#1}}

\newcommand \dswap [1] {\delta^{\sst{\textup{sw}}}_{#1}}
\newcommand \Dirr [1] {D_{\sst{\textup{ir}}}(#1)}
\newcommand \nfromo {\mathbb N^+}
\newcommand \nfromz {\mathbb N_0}
\renewcommand \phi {\varphi}
\newcommand \swsym {symmetric}
\newcommand \restrict [2] {{#1}\kern-1pt \rceil_{\kern-1pt #2}}
\newcommand \aut [1] {{#1}^{\sst{\textup{aut}}}}
\newcommand \ketparamaut [2] {{#1}_{#2}^{\sst{\textup{aut}}}}
\newcommand \daut [1] {\overline{\aut{#1}}}
\newcommand \wcond {\textup{(W)}}

\newcommand \ineq {ineq}
\newcommand \iref [1] {\textup{(\ineq\ref{#1})}}

\newcommand \what [1] {\widehat{#1}}
\newcommand \vonal {\noalign{\hrule}}
\newcommand \ssf [1] {{\sf{#1}}}

\newcommand \aindent  {35pt}
\newcommand \bindent  {45pt}
\newcommand \cindent  {13pt}
\newcommand \kr {\kern 1.0pt}

\newcommand \url [1] {\texttt{#1}}
\newcommand \pctex [1] {}
\newcommand\red [1] {\color{red}#1\color{black}}

\title{Symmetric embeddings of free lattices into each other}

\author[G.\ Cz\'edli]{G\'abor Cz\'edli}
\email{czedli@math.u-szeged.hu}
\urladdr{http://www.math.u-szeged.hu/~czedli/}
\address{University of Szeged, Bolyai Institute, 
Szeged, Aradi v\'ertan\'uk tere 1, HUNGARY 6720}

\author[G.\ Gyenizse]{Gerg\H{o} Gyenizse}
\email{gergogyenizse@gmail.com}
\urladdr{http://gllrumdsuhg.atw.hu/}
\address{University of Szeged, Bolyai Institute, 
Szeged, Aradi v\'ertan\'uk tere 1, HUNGARY 6720}

\author[\'A.\ Kunos]{\'Ad\'am Kunos}
\email{akunos@math.u-szeged.hu}
\urladdr{http://www.math.u-szeged.hu/~akunos/}
\address{University of Szeged, Bolyai Institute, 
Szeged, Aradi v\'ertan\'uk tere 1, HUNGARY 6720}

\thanks{This research was supported by
the Hungarian Research Grant KH 126581}

\subjclass{06B25{\hfill \red{Version: May 7, 2018}}}


\dedicatory{Dedicated to Ralph Freese and J.\ B.\ Nation on their seventieth birthdays}

\begin{document}

\maketitle

\begin{abstract}
By a 1941 result of Ph.\ M.\ Whitman, the free lattice $\Fl(3)$ on three generators
includes a sublattice $S$ that is isomorphic to the lattice $\Fl(\omega)=\Fl(\aleph_0)$ generated freely by denumerably many elements. 
The first author has recently ``symmetrized'' this classical result by constructing a sublattice $S\cong\fl \omega$ of $\fl 3$ such that $S$ is \emph{selfdually positioned} in $\fl 3$ in the sense that it is invariant under the natural dual automorphism of $\Fl(3)$ that keeps each of the three free generators fixed. Now we move to the furthest in terms of symmetry by constructing a selfdually positioned sublattice $S\cong\fl \omega$ of $\fl 3$ such that 
every element of $S$ is fixed by  all automorphisms of $\fl 3$. That is, in our terminology, we embed $\fl\omega$ into $\fl 3$ in a 
\emph{totally symmetric} way.  Our main result determines all pairs $(\kappa,\lambda)$ of cardinals greater than 2 such that  $\Fl(\kappa)$ is embeddable into $\Fl(\lambda)$ in a totally symmetric way. 
Also, we relax the stipulations on $S\cong\fl\kappa$ by requiring only that $S$ is closed with respect to the automorphisms of $\fl\lambda$, or $S$ is selfdually positioned and closed with respect to the automorphisms; we determine the corresponding pairs $(\kappa,\lambda)$ even in these two cases.  
We reaffirm some of our calculations with a computer program developed by the first author.
This program is for the word problem of free lattices, it runs under Windows, and it is freely available.
\end{abstract}

\section{Introduction and our results}

There are many nice and deep results on free lattices of the variety of all lattices. A large part of these results were achieved by Ralph~Freese and J.\,B.~Nation, to whom this paper is dedicated. Some of these results are included in \cite{fr85,fr87,frnat85,frnat16,nat82,nat84} and in the monograph Freese, Je\v zek, and Nation~\cite{fjnmonograph}, but this list is far from being complete. The monograph just mentioned serves as the reference book for the present paper.

By a classical result of Whitman \cite{whitman}, the free lattice $\Fl(\omega)=\Fl(\aleph_0)$ on denumerably many free generators  is isomorphic to a sublattice of the free lattice $\Fl(3)$ with three free generators. In fact, we know from a deep result of Tschantz~\cite{tschantz} that there are many copies of $\fl\omega$ in $\fl 3$; namely, every infinite interval of $\fl 3$ includes  a sublattice isomorphic to $\fl\omega$. 
For more about free lattices, the reader is referred to Freese, Je\v zek and Nation~\cite{fjnmonograph}.
In this paper, we embed free lattices into each other \emph{symmetrically}. 
For a free lattice $F$, 
\begin{equation}
\parbox{8cm}{$\delta=\delta_F$ will denote the \emph{natural} 
dual automorphism of $F$ that keeps the free generators fixed;}
\label{eqpbxnatduAl}
\end{equation}
it is uniquely determined.
A subset (or a sublattice) $S$ of $F$ is \emph{self\-dually positioned} if $\delta(S)=S$. Selfduality is a sort of symmetry, and a selfdually positioned sublattice is necessarily selfdual.
As the main result of \cite{czg-selfdual}, 
the first author proved that 
\begin{equation}
\parbox{6.8cm}{$\fl 3$ has a selfdually positioned sublattice that is isomorphic to  $\fl\omega$.}
\label{eqpbxflhrMgsD}
\end{equation}

Besides selfduality, there is a more general concept of symmetry, which is used even outside algebra; it is based on automorphisms. 
For a lattice $L$, let  $\Aut(L)$ denote the \emph{automorphism group} of $L$. We call a subset (or a sublattice) $S$ of $L$  \emph{\swsym{}} if $\pi(S)= \pi$ for every $\pi\in \Aut(F)$.  Also, an element $u\in L$ is a \emph{symmetric element} of  $L$ if $\set{u}$ is a \swsym{} subset of $L$. Note that there is no symmetric element in $\Fl(\omega)$. If $S$ contains only  symmetric elements of $L$, then $S$ is \emph{element-wise symmetric} in $L$.
Our key concept is the following.

\begin{definition}\label{deftotsyM}
A lattice embedding $\phi\colon L\to F$ of a lattice $L$ into a free lattice $F$
is  \emph{totally symmetric} if its range 
$\phi(L)=\set{\phi(u):u\in L}$
 is a selfdually positioned and element-wise symmetric sublattice of $F$.
\end{definition}

To expand our notation for all \emph{cardinal} numbers $\kappa\geq 3$, we denote by $\fl \kappa$ the free lattice with $\kappa$ many free generators. If $\kappa=n$ is a natural number, then we often write $\fl n$. Following the tradition, we often denote $\Fl(\aleph_0)$ by $\fl\omega$ even if $\omega$ is an \emph{ordinal} number. In order to  avoid ambiguity about natural numbers, we adhere to the notations  $\nfromo:=\set{1,2,3,\dots}$ and $\nfromz:=\set0\cup\nfromo$. The elements of $\nfromz$ are also cardinals; namely, the finite cardinal numbers. Our main result is the following.

\begin{theorem}\label{thmmain}~
\begin{enumerate}[\upshape (A)]
\item\label{thmmaina}
Assuming that $3\leq \kappa $ and  $3\leq \lambda$ are cardinal numbers, there exists a totally symmetric embedding  $\Fl(\kappa) \to \Fl(\lambda)$ if and only if $\lambda\in \nfromo$ is a natural number and  $\kappa\in\set{2k:k\in \nfromo}\cup\set{\aleph_0}$.
\item\label{thmmainb} In particular, there exists a totally symmetric embedding $\Fl(\omega) \to \Fl(3)$.
\end{enumerate}
\end{theorem}

For later reference, we mention the following corollary even if it  trivially follows from Theorem~\ref{thmmain}. 

\begin{corollary}\label{corolnhRm} There exists a totally symmetric embedding $\Fl(4) \to \Fl(3)$.
\end{corollary}

In addition to Theorem~\ref{thmmain} on total symmetry, we have some progress in studying selfdually positioned free sublattices, which is stated as follows.

\begin{theorem}\label{thmSD} Assuming that  $3\leq \kappa$ and $3\leq \lambda $ are cardinal numbers,  $\fl\lambda$ has a selfdually positioned sublattice  isomorphic to $\fl\kappa$
if and only if the inequality $\max\set{\kappa,\aleph_0} \leq \max\set{\lambda,\aleph_0}$ holds.
\end{theorem}

This theorem is stronger than \eqref{eqpbxflhrMgsD}, the main result of Cz\'edli~\cite{czg-selfdual}. Implicitly, 
the particular case of Theorem~\ref{thmSD} where $\lambda=3$ and  $\kappa$ belongs to $\set{\aleph_0}\cup\set{2k:k\in\nfromo}$ is also included in \cite{czg-selfdual}.

\subsection*{Prerequisites} The reader is expected to have  some basic familiarity only with the rudiments of lattice theory.
That is, only some preliminary sections of the monographs, say, 
 Gr\"atzer~\cite{gGLTF} or  Nation~\cite{nationbook} are assumed.
The results on free lattices that we need from the literature, mainly from Freese, Je\v zek, and Nation~\cite{fjnmonograph}, are known for most lattice theorists and will be quoted with sufficient details since the  paper is intended to be self-contained. In  Section~\ref{sectionfromKeyL}, we quote some recent achievements from \cite{czg-selfdual}; when reading this section, the reader does not have to but may want to  look into  Cz\'edli~\cite{czg-selfdual}\footnote{\red{Temporary note, only in the preprint version}: available from the author's homepage.}{} to verify how we quote from it.

\subsection*{Main ideas of the paper} In this subsection, we deal mainly with the totally symmetric embeddability of $\fl \omega$ into $\fl 3$, that is, with Part~\eqref{thmmainb} of Theorem~\ref{thmmain}; the rest of the results are derived from or proved like Theorem~\ref{thmmain}\eqref{thmmainb}, or they are easier. 

First, we define \emph{symmetric} elements $m_1<\dots<m_4$ in $\fl 3$, see \eqref{eqalignedXSRpPtW}, and  $\ovl {m_i}$ will be the dual of $m_i$ for $i\in\set{1,\dots,4}$. With some computation based on Whitman's condition, we can prove that 
\begin{enumerate}[\upshape({plan}1)]
{\setlength\itemindent{\cindent}
\item\label{pLaNna} $P=\set{m_1,\dots,m_4}\cup\set{\ovl{m_1},\dots,\ovl{m_4}}$ is the cardinal sum of two 4-element chains; see Figure~\ref{figConstr} for an illustration, and
}
{\setlength\itemindent{\cindent}\item\label{pLaNnb} we prove some properties of $P$ implying that 
the sublattice $[P]_{\fl 3}$  generated by $P$ in $\fl 3$ is isomorphic to the completely free lattice $\cfree(P;\leq)$ generated by the ordered set $(P;\leq)$.} 
\end{enumerate}
Combining the isomorphism $[P]_{\fl 3}\cong\cfree(P;\leq)$ with the main result of Rival and Wille~\cite{rivalwille}, we could immediately obtain an element-wise symmetric sublattice $S$ of $\fl 3$ such that $S\cong\fl\omega$. However, we want more. Hence,
\begin{enumerate}[\upshape({plan}1)]
\setcounter{enumi}{2}
{\setlength\itemindent{\cindent}\item\label{pLaNnc} we define $a,b\in [P]_{\fl 3}$ in \eqref{eqalignedXSRpPtW},  see also Figure~\ref{figConstr}, such that some computation based on Whitman's condition yields that 
$\fl 4$ is isomorphic to $[a,b,\ovl a,\ovl b]_{\fl 3}$, where $\ovl a$ and $\ovl b$ are the duals of $a$ and $b$, respectively. Note that the restriction of  $\delta_{\fl 3}$ to $\fl 4$ is not the natural dual automorphism of $\fl 4$ since it swaps the free generators of $\fl 4$. 
}
\end{enumerate}
At this stage, Corollary~\ref{corolnhRm} is already proved. 
Next, let  $\dswap4$ denote the unique
dual automorphism  of $\fl 4=\Fl(y_0,y_1,y_2,y_3)$ for which we have that $\dswap4(y_0)=y_1$,  $\dswap4(y_1)=y_0$,  $\dswap4(y_3)=y_4$,  $\dswap4(y_4)=y_3$; we call it the \emph{swapping} dual automorphism of $\fl 4$.
\begin{enumerate}[\upshape ({plan}1)]
\setcounter{enumi}{3}
{\setlength\itemindent{\cindent}\item\label{pLaNnd} By the ``diagonal method'' of Cz\'edli~\cite{czg-selfdual}, $\fl \omega$ is a sublattice of  $\fl 4$ closed with respect to $\dswap4$.}
\end{enumerate}
Finally,  a straightforward computation 
will show that if we combine (plan\ref{pLaNnc}) and (plan\ref{pLaNnd}), then their ``swapping'' features neutralize each other and we obtain  a totally symmetric embedding $\fl\omega\to\fl 3$, as required. 

If we embed $\fl\omega$ or $\fl\kappa$ into $\fl\lambda$, rather than into $\fl 3$,
then some of the above-mentioned computations, most of which can be done by a computer, become longer. Fortunately, we can often rely on the following fact, which deserves separate mentioning here: with two trivial exceptions, every symmetric element of $\fl\lambda$ is given by a near-unanimity term.

Let us note that the isomorphism $[P]_{\fl 3}\cong \cfree(P;\leq)$ and the above-mentioned result of Rival and Wille~\cite{rivalwille} 
are only motivating facts and will not be used in the detailed proof. Note also that this subsection will not be used in the rest of the paper; due to elaborated details and many internal references, 
the proofs are readable without keeping the main ideas in mind. Finally, since we also need to prove (our second) Theorem~\ref{thmSD}, we will prove more on $\set{a,b,\ovl a,\ovl b}$ than what is required by (plan\ref{pLaNnc}).

\subsection*{Outline} Our main results, Theorems~\ref{thmmain} and \ref{thmSD}, and our main ideas have already been presented; the rest of the  paper is structured as follows.  
We add some comments and two corollaries to the main result in 
Section~\ref{sectioncomments}. These corollaries characterize the pairs $(\kappa,\lambda)$ of cardinals having the property that there is an embedding $\fl\kappa\to \fl\lambda$ with symmetric range or with selfdually positioned and symmetric range.
The lion's share of our construction and (the Key) Lemma~\ref{lemmakeylemma} stating that this construction works are given in Section~\ref{sectionkeyLm}.  The Key Lemma  is proved in Section~\ref{sectionKLproofS}. Section~\ref{sectionfromKeyL} combines the construct given in Section~\ref{sectionkeyLm} with that given in Cz\'edli~\cite{czg-selfdual}. Section~\ref{sectionRestPrf} completes the proofs of our theorems and proves the corollaries.
Finally, Section~\ref{sectcompProgr} describes our computer program for the word problem of free lattices; note that this  program and its source file are freely available and the program proves Corollary~\ref{corolnhRm} in less than a millisecond.

\section{Remarks and corollaries}\label{sectioncomments}
We will often use the convenient notation $\fl\kappa=\Fl(x_i:i<\kappa)$ in order the specify the free generating set $\set{x_i:i<\kappa}$ of size $\kappa$; in this case, $i\in\nfromz$ denotes an ordinal number and $i<\kappa$ is understood as $|i|<\kappa$. 
(Equivalently, we could say that a cardinal number is the smallest ordinal with a given cardinality; then $i<\kappa$ would have its usual meaning for ordinals.)
Mostly, $\kappa$ is not larger than $\aleph_0$ and then $i\in\nfromz$ will denote a nonnegative integer. 
 Of course, we can write $\Fl(x_0,\dots,x_{\kappa-1})$ or $\Fl(y_i:i<\kappa)$ if $\kappa$ is finite or we need to denote the free generators in a different way, respectively. Also, we write $\fl X$ if we do not want to specify the size $|X|$ of the free generating set $X$.
An element $u$ of a lattice $L$ is \emph{doubly irreducible} if
$L\setminus\set u$ is closed with respect to both joins and meets,  that is, if $u$ is both \emph{join irreducible} and \emph{meet irreducible}. The set of doubly irreducible elements of $L$ will be denoted by $\Dirr L$. We know from Freese, Je\v zek, and Nation~\cite{fjnmonograph}, see
  Corollary 1.9 and the first sentence of the proof of Corollary 1.12,  that 
\begin{align}
&\Dirr{\fl\kappa} = \set{x_i:i<\kappa}, \text{ and}
\label{aligntxtDfkrZntNshG}\\
&\text{every element of }\fl\kappa\text{ is join or meet irreducible.}
\label{aligntxtkrdSkFrZnbT}
\end{align} 
Since a dual automorphism maps join-irreducible elements to meet irreducible ones and vice versa, \eqref{aligntxtDfkrZntNshG} and \eqref{aligntxtkrdSkFrZnbT} imply that for every dual automorphism $\psi$ of $\fl\kappa$, we have that
\begin{equation}
\set{u: \psi(u)=u}\subseteq \psi(\set{x_i:i<\kappa}) = \set{x_i:i<\kappa}.
\label{eqtxtFpNslmnTwszpThG}
\end{equation}

\begin{remark} The concept of totally symmetric embeddings 
might raise the question whether we could consider even stronger embeddings whose ranges are element-wise symmetric and are in \emph{element-wise} selfdual positions. We obtain from \eqref{eqtxtFpNslmnTwszpThG} that the answer is negative, since  at most the free generators are in element-wise selfdual positions and they form an antichain. This justifies our terminology to call the embeddings in Theorem~\ref{thmmain} \emph{totally} symmetric.
\end{remark}

\begin{remark} Assume that $3\leq\kappa\leq\aleph_0$ and $3\leq\lambda\leq\aleph_0$; as a comparison between the result of Whitman~\cite{whitman} and Theorem~\ref{thmmain}, note the following. It follows immediately from Whitman' result that 
$\fl\kappa$ is embeddable into $\fl\lambda$, because embeddability is a transitive relation and $\gamma_1\leq\gamma_2$ implies that  $\fl{\gamma_1}$ is embeddable into $\fl{\gamma_2}$. However, the analogous implication fails for totally symmetric embeddability, since a symmetric element of $\fl{\gamma_1}$ is not symmetric in $\fl{\gamma_2}$ for $\gamma_2>\gamma_1$. This explains that, as opposed to Whitman's result,  Theorem~\ref{thmmain} contains two parameters, $\kappa$ and $\lambda$.
\end{remark}

For a lattice $L$, let $\Daut(L)$ be the set of all automorphisms and dual automorphisms of $L$.
As a consequence of  \eqref{aligntxtDfkrZntNshG} and \eqref{aligntxtkrdSkFrZnbT}, note that for $\kappa\leq \omega$,
\begin{equation}
\parbox{9.5cm}{$\Dirr{\fl\kappa}$ is closed with respect to every $\pi\in\Daut(\fl\kappa)$. Furthermore, each $\pi\in\Daut(\fl\kappa)$ is determined by its restriction to $\Dirr{\fl\kappa}$ if we know whether it is an automorphism or a dual automorphism.}
\label{eqpbxdnwlmghTrzzdtnbQ}
\end{equation}
With respect to composition, $\Daut(L)$ is a group and $\Aut(L)$ is a normal subgroup in it with index $[\Daut(L):\Aut(L)]=2$. Let us call a subset $S$ of $L$ a \emph{\DAstr-symmetric} subset if $\pi(S)=S$ for all $\pi\in\Daut(L)$. 
A dually positioned and element-wise symmetric sublattice of $\fl\lambda$, like the range of a totally symmetric embedding, 
is clearly \DAstr-symmetric but not conversely. 
This might give some hope that a counterpart of 
Theorem~\ref{thmmain} for embeddings with \DAstr-symmetric
ranges allows the case when $\kappa$ is an odd natural number. However, the following corollary of Theorem~\ref{thmmain} shows that this is not so if $\kappa\neq\lambda$. This  corollary as well as Corollary~\ref{corolaUtsym}  will be proved in Section~\ref{sectionRestPrf}.

\begin{corollary}\label{corolDA}
Assuming that $3\leq \kappa$ and    $3\leq \lambda$ are  cardinal numbers, the following two conditions are equivalent.
\begin{enumeratei}
\item\label{corolDAa} There exists an embedding  $\Fl(\kappa) \to \Fl(\lambda)$ with \DAstr-symmetric range.
\item\label{corolDAb} Either $\kappa=\lambda$, or we have that $\lambda\in\nfromo$ and $\kappa\in\set{2k:k\in\nfromo}\cup\set{\aleph_0}$.
\end{enumeratei}
\end{corollary}

The situation is different if we deal with embeddings whose ranges are symmetric with respect only to $\Aut{\fl\lambda}$.

\begin{corollary}\label{corolaUtsym}
Assuming that $3\leq \kappa$ and    $3\leq \lambda$ are  cardinal numbers, there exists an embedding  $\Fl(\kappa) \to \Fl(\lambda)$ with symmetric range if and only if 
either $\lambda\in\nfromo$ and $\kappa\leq\aleph_0$, or $\kappa=\lambda\geq \aleph_0$.
\end{corollary}

Since our concepts are based on the automorphisms of $\fl n$, where $n:=\lambda$ is a positive integer, let us have a look at what these automorphisms and the elements of $\Daut(\fl n)$ are. Let $\Sym_n=\Sym(\set{0,1,\dots,n-1})$ denote the group of all permutations of $\set{0,1,\dots,n-1}$ with respect to composition, and let $\ssf C_2=\set{1,-1}$ be the cyclic group of order 2, considered a subgroup of the group $(\mathbb R\setminus\set0;\cdot)$ of nonzero real numbers with respect to multiplication. Using that  $\fl n=\Fl(x_i:i<n)$ is freely generated by the set $\set{x_i:i<n}$, the following remark is straightforward and its proof is omitted.

\begin{remark} 
\label{remcspRtlrS}
For $n\in\nfromo$ and the free lattice $\fl n=\Fl(\set{x_i:i<n})$, the groups $\Aut(\fl n)$ and $\Daut(\fl n)$ are isomorphic to $\Sym_n$ and $\Sym_n\times \ssf C_2$, respectively. More specifically, for $\sigma\in\Sym_n$, let $\aut\sigma\colon \fl n\to \fl n$ and $\daut\sigma\colon \fl n\to \fl n$ be the maps defined by 
\begin{align*}
\aut\sigma(t(x_0,\dots, x_{n-1}))&:= t(x_{\sigma(0)},\dots, x_{\sigma(n-1)})\text{ and}\cr
\daut\sigma(t(x_0,\dots, x_{n-1}))&:= \delta_{\fl n}(t(x_{\sigma(0)},\dots, x_{\sigma(n-1)})),
\end{align*}
respectively, where $t$ denotes an $n$-ary lattice term and  $\delta=\delta_{\fl n}$ has been defined in  \eqref{eqpbxnatduAl}. Then the maps
\begin{align*}
&\Sym_n\to \Aut(\fl n),\text{ defined by }\sigma\mapsto \aut\sigma\text{, and}\cr
&\Sym_n\times \ssf C_2\to \Daut(\fl n),\text{ defined by }
(\sigma,k)\mapsto
\begin{cases}
\daut\sigma,&\text{if }k=-1,\cr
\aut\sigma,&\text{if }k=1
\end{cases}
\end{align*}
are group isomorphisms. (In particular, they are well defined maps.)
\end{remark}

\section{Construction and the Key Lemma}\label{sectionkeyLm}
\subsection{Notation}\label{subsectkla}
The elements of a free lattice $\fl\kappa=\Fl(x_i:i<\kappa)$ will be represented by \emph{lattice terms} over the set $\set{x_i:i<\kappa}$ of variables. Although there are many terms representing the same element of $\fl\kappa$, it will not cause any confusion that 
\begin{equation}
\parbox{10cm}{we often  treat and  call these \emph{terms as elements} of the free lattice;}
\label{eqpbxtrTtRlMnts}
\end{equation}
see pages 10--11 in Freese, Je\v zek and Nation~\cite{fjnmonograph} for a more rigorous setting. 
The dual of a term $t$ will always be denoted by $\ovl t$; the overline reminds us that dualizing at visual level  means to reflect Hasse diagrams across horizontal axes. The Symmetric part of the Free Lattice of $\fl \kappa$ will be denoted as follows; capitalization explains the acronym:
\begin{equation*}
\Sl(\kappa):=\set{u\in \fl\kappa:  \pi(u)=u\text{ for all }\pi\in\Aut{\fl\kappa} }.
\end{equation*}
Clearly, $\Sl(\kappa)$ is a sublattice of $\fl\kappa$; this fact will often be used implicitly.

\subsection{Constructing some important terms}\label{subsectklb} In this subsection,
we give a construction for the particular case $(\kappa,\lambda)=(4,n)$ of Theorem~\ref{thmmain}\eqref{thmmaina}. 
Let us agree to the following conventions. The set  
\begin{equation*}
\text{$\set{0,1,\dots,n-1}$ will also be denoted by $[n)$.}
\end{equation*} 
The inequality $i<n$ is equivalent to $i\in [n)$. 
Whenever  $x_i$, $x_j$, etc.\ refer to a free generator  of $\Fl(n)=\Fl(x_0,\dots, x_{n-1})$, then $i,j,\dots$ will automatically belong to $[n)$; this convention will often save us from indicating, say, that $i<n$ or $i,j\in[n)$  below the $\bigvee$ and $\bigwedge$  operation signs. Also, we frequently abbreviate the conjunction of $i\in [n)$ and $j\in[n)\setminus\set i$ by the short form $i\neq j$, and self-explanatory similar other abbreviations may also occur. 
For the rest of this  section, let $3\leq n=\lambda \in\nfromo$ and  
\[
\fl\lambda=\Fl(n)=\Fl(x_0, \dots, x_{n-1})=\Fl(x_i: i<n).
\] 
By induction on $j$, we define the following $n$-ary lattice terms over the set $\set{x_i:i<n}$ of variables; according to \eqref{eqpbxtrTtRlMnts}, they will also be considered elements of $\fl n$. Namely, we let
\begin{equation}
\begin{aligned}
p^{(i)}_0&:=x_i\text{ for }i\in\set{0,1,\dots,n-1}=[n),\cr
p^{(i)}_j&:=x_i\vee\bigvee_{\substack{ i_1<i_2 \\  i_1,i_2\in[n)\setminus\set i }}  
  \bigl(p^{(i_1)}_{j-1}\wedge p^{(i_2)}_{j-1} \bigr)\,\, \text{ for }i\in[n)\text{ and }j\in\nfromo,\cr
m_j&:=\bigvee_{{ i_1<i_2,\,\, i_1,i_2\in[n)}} \bigl( p^{(i_1)}_j \wedge p^{(i_2)}_j \bigr)\,\,\text{ for }j
\in\nfromz, \cr
a&:=m_1\vee \ovl{m_3},\,\,\text{ and }\,\, b:=m_2 \vee \ovl{m_4}.
\end{aligned}
\label{eqalignedXSRpPtW}
\end{equation}
For later reference, we note that the set
\begin{equation}
\set{a,\ovl a, b, \ovl b}
\label{eqshhbkvlhkkvlb}
\end{equation}
will play an important role in the paper. 
We say that a subset $X$ of a lattice \emph{freely generates} if the sublattice $S$ generated by $X$ is a free lattice with $X$ as the set of free generators. Next, we formulate our Key Lemma, which is stronger than asserting that the set in \eqref{eqshhbkvlhkkvlb} freely generates.  The proof of the Key Lemma will be postponed to Section~\ref{sectionKLproofS}.

\begin{lemma}[Key Lemma]\label{lemmakeylemma}
If $3\leq n\in\nfromo$, then 
the elements $m_j$ and $\ovl{m_j}$ for $j\in\nfromz$, $a$, $b$, $\ovl{a}$, and $\ovl{b}$ all belong to $\Sl(n)$. Furthermore, 
$\set{a,\ovl a, b, \ovl b, x_0}$ is a five-element subset of $\Sl(n)$ that freely generates. 
\end{lemma}

For later reference, based on Remark~\ref{remcspRtlrS}, note the following trivial lemma.

\begin{lemma}\label{lemmakshbzRmBgJwWm}
For every $i\in [n)$, $j\in\nfromz$ and $\sigma\in\Sym([n))$, we have that $\aut\sigma(p_j^{(i)}) = p_j^{(\sigma(i))}$ and $\aut\sigma (m_j) = m_j$.
\end{lemma}

\begin{proof}[Proof of Lemma~\ref{lemmakshbzRmBgJwWm}]
The first equality above follows from the fact that in \eqref{eqalignedXSRpPtW},  each stipulation of the form $i_1<i_2$ can be replaced by $i_1\neq i_2$. The second equality follows from the first one.
\end{proof}

\section{The proof of the Key Lemma}\label{sectionKLproofS}
From the theory of free lattices, we only use three basic facts, 
which we recall below as lemmas; all of them can be found  in Freese, Je\v zek and Nation~\cite{fjnmonograph}. An element $u$ of a lattice $L$ is \emph{join prime} if for all $k\in\nfromo$ and $x_0,\dots, x_{k-1}\in L$, the inequality $u\leq x_0\vee\dots\vee x_{k-1}$ implies that $u\leq x_i$ for some $i\in[k)$. Meet prime elements are defined dually. An element is \emph{doubly prime} if it is join prime and meet prime.

\begin{lemma}[{Freese, Je\v zek and Nation~\cite[Corollary 1.5]{fjnmonograph}}]\label{lemmagnJMggDprm} 
In every free lattice $\fl X$, the free generators are doubly prime elements.
\end{lemma}

The following statement says that free lattices satisfy \emph{Whitman's condition} \wcond{}, see Whitman~\cite{whitman}.

\begin{lemma}[{Freese, Je\v zek and Nation~\cite[Theorem 1.8]{fjnmonograph}}]\label{lemmaWcnDtLgnB} 
For arbitrary elements $u_1,\dots,u_r,v_1,\dots,v_s $ of a free lattice  $\Fl(X)$,
\begin{enumerate}
{\setlength\itemindent{0.5pt}
\item[\wcond] The inequality 
$u=u_1 \wedge\dots\wedge u_r \leq v_1 \vee \dots \vee v_s = v$  implies that either $u_i\leq v$ for some subscript $i$, or $u \leq v_j$ for some $j$.
}
\end{enumerate} 
\end{lemma}

Next, we describe whether a subset of a free lattice generates freely or not.

\begin{lemma}[{Freese, Je\v zek and Nation~\cite[Corollary 1.13]{fjnmonograph}}]\label{lemmafjnFrGSSt}
A nonempty subset $Y$ of $\Fl(X)$ generates freely if and only if for all $h\in Y$ and all finite subsets $Z\subseteq Y$, the following condition and its dual hold.
$$h\notin Z \;\; \text{ implies }\;\; h\nleq \bigvee_{z\in Z}z.$$
\end{lemma}

The general assumption for the rest of this section is
that $3\leq n\in\nfromo$ and $\fl n=\Fl(x_0,\dots, x_{n-1})=\Fl(x_i: i<n)$. 
A (lattice) term $t$ is called a {\it near-unanimity lattice term} or, shortly, an {\it NU-term} if it satisfies
\[
t(y, x, \dots, x)=t(x,y,x, \dots, x)= \dots = t(x, \dots, x,y)=x.
\]
Since the lattice operations are idempotent, it is obvious that
\begin{equation}
\text{the join and the meet of two $n$-ary NU-terms are NU-terms.}
\label{eqtxtjmtwNUtrms}
\end{equation}
If $t_1$ and $t_2$ are $n$-ary lattice terms such that $t_1=t_2$ in $\fl n$, see \eqref{eqpbxtrTtRlMnts}, then $t_1\in\Sl(n)$ iff $t_2\in \Sl(n)$. Also, for $t_1=t_2\in\fl n$,  $t_1$ is an NU-term iff so is $t_2$. So convention \eqref{eqpbxtrTtRlMnts} still applies.

\begin{lemma}
\label{lemmasymNU} If $t\in\Sl(n)=\Sl(x_0,\dots, x_{n-1})$ such that none of the equalities
$t(x_0,\dots, x_{n-1})=\bigwedge_{i\in[n)}x_i$ and $t(x_0,\dots, x_{n-1})=\bigvee_{i\in[n)}x_i$ holds in $\Fl(n)$, then $t$ is a near-unanimity term. 
\end{lemma}

\begin{proof} First, we are going to show that%
\begin{align}
\parbox{8.5cm}{if $g\in \Sl(n)$  and $x_i\leq g$ in $\fl n$ for some $i\in[n)$, then $g=1_{\Sl(n)}=x_0\vee\dots\vee x_{n-1}$, and dually.}
\label{eqpbxfhszRtdzBg}
\end{align}
Assume that $x_i\leq g$. For $i'\in[n)$, take a permutation $\sigma\in\Sym([n))$ with $\sigma(i)=i'$, then $x_{i'}=\aut\sigma(x_i)\leq\aut\sigma(g)=g$, which immediately gives that 
$1_{\Sl(n)}=\bigvee_{i'\in[n)}x_{i'} \leq g\leq 1_{\Sl(n)}$, implying \eqref{eqpbxfhszRtdzBg}. Next,
assume that $t$ satisfies the assumptions of the lemma. As a binary lattice term, $t(x,\dots,x,y)$  equals one of  $x$, $y$, $x\wedge y$, and $x\vee y$ in $\Fl(x,y)$.
If we had that, in $\Fl(x,y)$,  $t(x,\dots,x,y)=x\wedge y$, then 
\begin{align*}
t(x_0,\dots,x_{n-1})
&\leq t(x_0\vee \dots\vee x_{n-2},\dots,x_0\vee \dots\vee x_{n-2},x_{n-1})\cr
&\leq (x_0\vee \dots\vee x_{n-2})\wedge x_{n-1}\leq x_{n-1}
\end{align*}
together with \eqref{eqpbxfhszRtdzBg} would yield that $t=0_{\Sl(n)}=x_0\wedge \dots\wedge x_{n-1}$, a contradiction.  
Hence, $t(x,\dots,x,y)$ is distinct from $x\wedge y$, and it is distinct also from $x\vee y$ by duality. The case $t(x,x,\dots,x,y)=y$ is impossible, because it would imply that $
1=t(0,\dots,0,1)\leq t(0,1,\dots,1)=t(1,\dots,1,0)=0$ holds in the two-element lattice $\mathbf{2}$, which is a contradiction. Hence, $t(x,x,\dots,x,y)=x$, which means that $t$ is an NU-term since it is symmetric.
\end{proof}

\begin{figure}[th]
\includegraphics [scale=0.95] {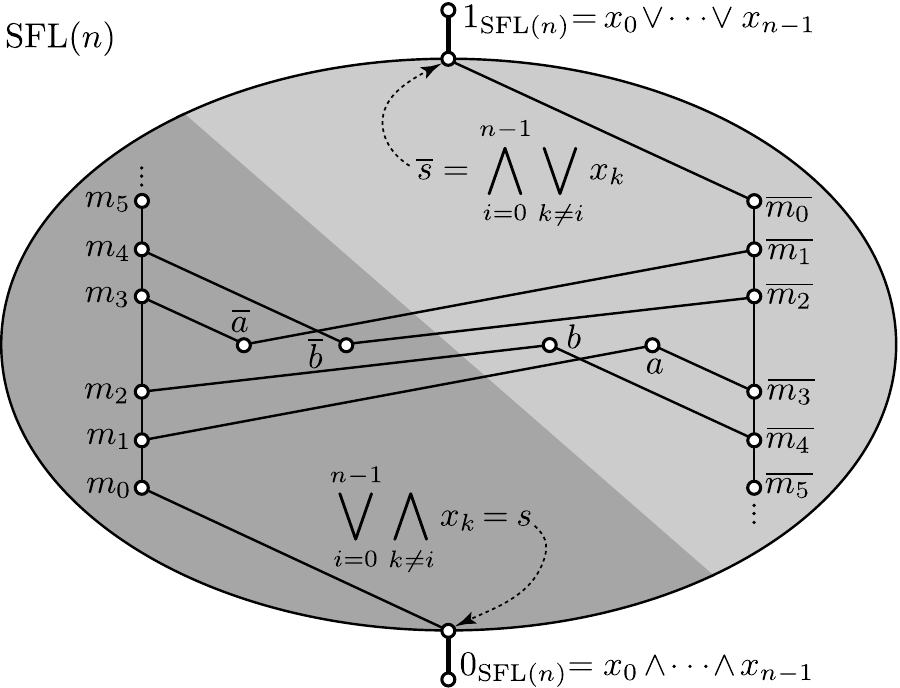}
\caption{The lattice $\Sl(n)$ for $n>3$}
\label{figConstr}
\end{figure}

\begin{lemma}
\label{lemmabottopNu} 
There is exactly one  atom in $\Sl(n)$, and it  is 
\[s:=\bigvee_{i\in[n)}\bigwedge_{i'\in [n)\setminus\set i}x_{i'}\,.
\] 
The only coatom of $\Sl(n)$ is $\ovl{s}$. Except from its bottom 
$0_{\Sl(n)}=\bigwedge_{i<n} x_i$ and  top $1_{\Sl(n)}=\bigvee_{i<n} x_i$, every element of  $\,\Sl(n)$ is in the interval $[s,\ovl{s}]$.
\end{lemma}

The statement of this lemma for $3<n\in\nfromo$ is illustrated by Figure~\ref{figConstr}, where only the two thick edges stand for coverings in $\Sl(n)$; the thin lines indicate comparabilities that need not be coverings. (These comparabilities will be proved later; see Lemma~\ref{lemmam1notsubdual}.)  The reflection across the symmetry center point, which is not indicated in the figure, represents the restriction of  $\delta_{\fl n}$ to $\Sl(n)$. We could obtain a similar figure for $n=3$ by removing the vertices $m_0$ and $\ovl{m_0}$ and decreasing the subscripts of the remaining vertices by $1$; see Lemma~\ref{lemmam1notsubdual}.

\begin{proof}
By Lemma \ref{lemmasymNU} and the duality principle, we need to show  only that $t\geq s$ holds for every near-unanimity term $t\in\Sl(n)$. This follows from the fact that for each $i\in[n)$, 
\begin{align*}
t&(x_0,\dots,x_{i-1},x_i,x_{i+1},\dots x_{n-1}) \cr
&\geq 
t(\bigwedge_{i'\neq i}x_{i'},\dots,\bigwedge_{i'\neq i}x_{i'},x_i,\bigwedge_{i'\neq i}x_{i'},\dots,\bigwedge_{i'\neq i}x_{i'})=\bigwedge_{i'\neq i}x_{i'}.\qedhere
\end{align*}
\end{proof}

\begin{figure}[ht]
\includegraphics [scale=0.95] {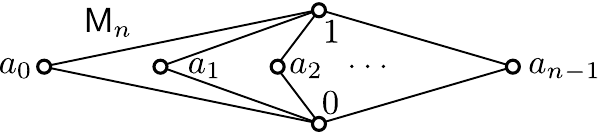}
\caption{The lattice $\ssf M_n$}
\label{figMn}
\end{figure}

In order to get some preliminary insight into $\Sl(n)$, note the following. As usual,  $\ssf M_n$ denotes $(n+2)$-element lattice given in Figure \ref{figMn}. 
For a permutation $\pi\colon [n)\to [n)$, let $\pi_X$ and $\pi_A$ denote the permutation of $\set{x_i:i<n}$ and that of $\set{a_i:i<n}$ defined by $\pi_X(x_i)=x_{\pi(i)}$ and  $\pi_A(a_i)=a_{\pi(i)}$ for all $i<n$, respectively.  These permutations uniquely extend to automorphisms $\aut\pi_X\in\Aut({\Sl(n)})$ and $\aut\pi_A\in\Aut(\ssf M_n)$, respectively. Consider the  natural homomorphism
\begin{equation}
\eta\colon \fl n\to \ssf M_n \text{ defined by }\eta(t(x_0,\dots,x_{n-1}))=t(a_0,\dots,a_{n-1}).
\label{eqnatHmTdvsz}
\end{equation}
Clearly, $\aut\pi_A\circ\eta=\eta\circ\aut\pi_X$. This implies easily that the $\eta$-image of a symmetric element of $\fl n$ is  a symmetric element of $\ssf M_n$. But the symmetric elements of $\ssf M_n$ form $\alg 2$ (as a sublattice), and we obtain that $\restrict \eta{\Sl(n)}\colon \Sl(n)\to \tbf 2$ is a surjective homomorphism. So the kernel of $\restrict \eta{\Sl(n)}$ cuts $\Sl(n)$ into a prime ideal, the dark-grey lower part together with the bottom element in Figure \ref{figConstr}, and its complementary  prime filter, the light-grey upper part together with the top in the figure. Besides that their shades are distinct in the Figure, the two parts are separated by a dashed line.

\begin{lemma}
\label{lemmapjincreasing}
For every $i<n$, the sequence $\set{p^{(i)}_j: j\in \nfromz}$ is strictly increasing, that is, $p^{(i)}_0< p^{(i)}_1 < p^{(i)}_2
<\dots$  in $\fl n$.
\end{lemma}

\begin{proof} 
First, we show by induction on $j$ that
\begin{equation}
\text{$p^{(r)}_j \wedge p^{(s)}_j\nleq x_k$ for all $j\in\nfromo$ and $k,r,s\in[n)$.}
\label{eqtxtxcvxcei}
\end{equation}
Suppose, for a contradiction, that  \eqref{eqtxtxcvxcei} fails, and let $j\in\nfromo$ be the smallest number violating it. Pick $k,r,s\in [n)$ such that $p^{(r)}_j \wedge p^{(s)}_j\leq x_k$. Since $x_k$ is meet prime by Lemma~\ref{lemmagnJMggDprm}, we can assume that $p^{(r)}_j\leq x_k$. 
By \eqref{eqalignedXSRpPtW}, $x_r\leq x_k$, whence $r=k$ and we have that $p^{(k)}_j\leq x_k$. Pick $r'<s'$ in $[n)\setminus\set k$; this is possible since $n\geq 3$. Since  $p^{(k)}_j\leq x_k$,  \eqref{eqalignedXSRpPtW} gives that  $p^{(r')}_{j-1} \wedge p^{(s')}_{j-1}\leq x_k$. Since this inequality does not violate  \eqref{eqtxtxcvxcei} by the choice of $j$, it follows that $j-1\notin\nfromo$. Hence, $j=1$ and \eqref{eqalignedXSRpPtW} turns $p^{(r')}_{j-1} \wedge p^{(s')}_{j-1}\leq x_k$ into $x_{r'}\wedge x_{s'}\leq x_k$, which is a contradiction since $k\notin\set{r',s'}$. This proves \eqref{eqtxtxcvxcei}.

Next, based on \eqref{eqalignedXSRpPtW}, a trivial induction on $j$ shows that  
\begin{equation}
\text{for all $j\in\nfromz$ and $i\in[n)$, $\,\,p^{(i)}_j(a_0,\dots,a_{n-1})=a_i$ holds in $\ssf M_n$.}
\label{eqtxtdhbmzTwMn}
\end{equation}
This implies  that, for any $j,j'\in\nfromz$ and $i,i'\in[n)$,
\begin{equation}
\text{if $i\neq i'$, then the terms $p^{(i)}_j$ and $p^{(i')}_{j'}$ are incomparable in $\fl n$.}
\label{eqtxtwqoiropt} 
\end{equation}
A straightforward induction yields that the sequence is increasing, that is,   
\begin{equation}
\text{
$p^{(i)}_0\leq p^{(i)}_1 \leq p^{(i)}_2 \leq \dots$ holds in $\fl n$.}
\label{eqtxtsrtLmnTnNdzTs}
\end{equation} 

Armed with the preparations above,  it suffices to prove the strict inequalities in the lemma  only for $i=0$, since then the case $i>0$ will follow by symmetry.
For the sake of contradiction,  suppose that there exists a $j\in \nfromz$ such that
\begin{equation}
p_{j+1}^{(0)}\overset{\eqref{eqalignedXSRpPtW}}=x_0\vee \bigvee_{i_1<i_2,\,\,i_1,i_2\in[n)\setminus\set 0}\bigl(p^{(i_1)}_{j} \wedge p^{(i_2)}_{j}\bigr)\leq p^{(0)}_j.
\label{qwekuezh}
\end{equation}
Let $j$ be minimal with this property. By symmetry, the superscript 0 does not play a distinguished role here. That is, until the end of the proof,
\begin{equation}
\text{for each $i\in[n)$, $j\in\nfromz$ is minimal such that  $p_{j+1}^{(i)}\leq p_{j}^{(i)}$.} 
\label{eqtxtjmFrLnJfkLtt}
\end{equation}
We are going to derive a contradiction  from \eqref{qwekuezh} by infinite descent.
Since $p^{(0)}_1=x_0\vee (x_1\wedge x_2)\vee \dots \leq x_0=p^{(0)}_0$ would lead to $x_1\wedge x_2\leq x_0$, which fails even in $\alg 2$, we obtain that $j > 0$. So, as the first step of the descent, we conclude that $j-1\in\nfromz$. 
Thus,  \eqref{qwekuezh} and $j-1\in\nfromz$ imply that
\begin{equation}
p^{(1)}_j \wedge p^{(2)}_j\leq 
p^{(0)}_j \overset{\eqref{eqalignedXSRpPtW}}=
x_0+\bigvee_{i_1<i_2,\,\,i_1,i_2\in[n)\setminus\set 0} \bigl(p^{(i_1)}_{j-1} \wedge p^{(i_2)}_{j-1}\bigr).
\label{eqou9eqw}
\end{equation}
By \eqref{eqtxtwqoiropt}, this inequality would fail if one of the two meetands on the left was omitted. Hence, \wcond{} and \eqref{eqtxtxcvxcei} yield that 
 $p^{(1)}_j \wedge p^{(2)}_j\leq  p^{(u_1)}_{j-1} \wedge p^{(v_1)}_{j-1}$ for some $v_1, u_1\in[n)\setminus\set0$. 
In particular, $p^{(1)}_j \wedge p^{(2)}_j\leq  p^{(u_1)}_{j-1}$. We formulate this inequality together with $j-1\in\nfromz$ also in the following way: the condition 
\begin{equation}
\text{$j- m\in\nfromz$ and there is a $u_m\in[n)$ such that $p^{(1)}_j \wedge p^{(2)}_j\leq  p^{(u_m)}_{j-m}$}
\label{eqNtfwTGtsp}
\end{equation}
holds for $m=1$. 
In order to continue the descent in infinitely many steps, we assert that
\begin{equation}
\parbox{5.4cm}{for every $m\in\nfromo$, if \eqref{eqNtfwTGtsp} holds for $m$, then it holds also for $m+1$.}
\label{eqpbxdhTrjNmTrsF}
\end{equation}
In order to prove \eqref{eqpbxdhTrjNmTrsF}, assume  \eqref{eqNtfwTGtsp}  for $m$. The case $j-m=0$ is ruled out by \eqref{eqtxtxcvxcei}, whence $j-(m+1)\in\nfromz$. Hence, the inequality in  \eqref{eqNtfwTGtsp} gives that 
\begin{equation}
 p^{(1)}_j \wedge p^{(2)}_j\leq  p^{(u_m)}_{j-m}\overset{\eqref{eqalignedXSRpPtW}}=
 x_{u_m}\vee\bigvee_{\substack{i_1<i_2\\i_1,i_2\in[n)\setminus\set {u_m}}} \bigl(p^{(i_1)}_{j-(m+1)} \wedge p^{(i_2)}_{j-(m+1)}\bigr).
\label{eqfsRrgtcWx}
\end{equation}
As before, we are going to apply \wcond{} to \eqref{eqfsRrgtcWx}; however, the argument for the meetands on the left is a bit longer. If $u_m\notin\set{1,2}$, then we obtain from \eqref{eqtxtwqoiropt} that  none of the meetands on the left of  \eqref{eqfsRrgtcWx}  can be omitted from  the inequality.
If $u_m=1$, then $p^{(1)}_j$ still cannot be omitted by \eqref{eqtxtwqoiropt}, but we need the same fact for the other meetand, $p^{(2)}_j$. Observe that if $p^{(2)}_j$ was omitted and $u_m$ equaled $1$, then we would have by \eqref{eqtxtsrtLmnTnNdzTs} that $p^{(1)}_j \leq p^{(1)}_{j-m}  \leq p^{(1)}_{j-1}$, contradicting 
\eqref{eqtxtjmFrLnJfkLtt}. So none of the two meetands in question can be omitted if $u_m=1$, and the same is true for $u_m=2$ since $1$ and $2$ play symmetric roles.
This shows that no matter what is $u_m$,  none of the two meetands on the left of \eqref{eqfsRrgtcWx} can be omitted. 
Therefore, \eqref{eqfsRrgtcWx}, \wcond{} and  \eqref{eqtxtxcvxcei}  imply that 
$p^{(1)}_j \wedge p^{(2)}_j\leq  p^{(u_{m+1})}_{j-(m+1)} \wedge p^{(v_{m+1})}_{j-(m+1)}$ for some $u_{m+1}, v_{m+1}\in[n)\setminus\set{u_m}$. 
In particular, $p^{(1)}_j \wedge  p^{(2)}_j\leq   p^{(u_{m+1})}_{j-(m+1)}$. Thus, we conclude that \eqref{eqNtfwTGtsp} holds for $m+1$, completing the proof of \eqref{eqpbxdhTrjNmTrsF}. Consequently, \eqref{eqNtfwTGtsp} holds for all $m\in\nfromo$, which contradicts the finiteness of $j\in\nfromz$ and completes the proof of Lemma~\ref{lemmapjincreasing}.
\end{proof}

\begin{lemma}
\label{lemmamjisincreasing}
The sequence $\set{m_j: j\in \nfromz}$ is strictly increasing, that is, we have that $m_0< m_1 < m_2 <\dots$  in $\fl n$.
Also, $\set{\ovl{m_j}: j\in \nfromz}$ is strictly decreasing.
\end{lemma}

\begin{proof} It suffices to deal only with  $\set{m_j: j\in \nfromz}$. The sequence in question is increasing by its definition, see \eqref{eqalignedXSRpPtW}, and Lemma \ref{lemmapjincreasing}. For the sake of contradiction, suppose that $m_{j}\leq m_{j-1}$ holds for some $j\in\nfromo$. 
Then, since  all joinands of $m_{j}$ are less than or equal to $m_{j-1}$, we have that, in particular, 
\begin{equation}
p^{(1)}_{j} \wedge p^{(2)}_{j}\leq m_j=\bigvee_{i_1<i_2,\,\,i_1,i_2\in[n)}\bigl(p^{(i_1)}_{j-1} \wedge p^{(i_2)}_{j-1}\bigr).
\label{eqbmhTrzBvmJ}
\end{equation}
It follows from \eqref{eqtxtdhbmzTwMn} that, with  $\vec a=(a_0,\dots, a_{n-1})\in \ssf M_n\times\dots\times \ssf M_n$, we have that $m_j(\vec a)=0$ while $p^{(1)}_{j}(\vec a)=a_1$ and $p^{(2)}_{j}(\vec a)=a_2$. Hence, neither of the meetands on the left of \eqref{eqbmhTrzBvmJ} can be omitted without breaking the inequality. Thus, applying \wcond{} to \eqref{eqbmhTrzBvmJ}, it follows that 
$p^{(1)}_{j} \wedge p^{(2)}_{j}\leq p^{(u_1)}_{j-1} \wedge p^{(v_1)}_{j-1}$
for some $u_1,v_1\in[m)$. This yields that 
$p^{(1)}_{j} \wedge p^{(2)}_{j}\leq p^{(u_1)}_{j-1}$. Therefore, \eqref{eqNtfwTGtsp} holds for $m=1$. Consequently, it follows from  \eqref{eqpbxdhTrjNmTrsF} that  \eqref{eqNtfwTGtsp}  holds for all $m\in\nfromo$, which is a contradiction since $j-m\notin \nfromz$ for $m>j$.
\end{proof}

The following lemma states something on $\Sl(n)$, not on $\Fl(n)$.

\begin{lemma}
\label{lemmamjisdoubleprime}
For all $j\in\nfromz$, $m_j$ and $\ovl{m_j}$ are  doubly prime elements of $\Sl(n)$. 
\end{lemma}

\begin{proof}
By duality, it suffices to deal only with $m_j$. 
In order to show that $m_j$ is join prime, assume that
$m_j\leq h_1\vee h_2$ where $h_1,h_2\in\Sl(n)$. Remember that the containment here means that $h_1$ and $h_2$ are fixed points of every automorphism of $\fl n$. We have to show that $m_j\leq h_i$ for some $i\in\set{1,2}$. There are two cases to consider. 
First, assume that 
\begin{equation}
\text{
there exists an $i\in\set{1,2}$ such that $p^{(0)}_j  \wedge  p^{(1)}_j\leq h_i$.}
\label{eqpbxdPSWrmXvQpkjsbtz}
\end{equation}
In this case, for each $(u_1,u_2)\in[n)\times [n)$ with $u_1<u_2$, pick a permutation $\sigma\in\Sym([n))$ such that $\sigma(0)=u_1$ and $\sigma(1)=u_2$. By Remark~\ref{remcspRtlrS} and Lemma~\ref{lemmakshbzRmBgJwWm},
\begin{equation}
\begin{aligned}
p^{(u_1)}_j  \wedge p^{(u_2)}_j&=p^{(\sigma(0))}_j \wedge p^{(\sigma(1))}_j\cr
&=\aut\sigma(p^{(0)}_j \wedge p^{(1)}_j)\leq \aut\sigma(h_i)=h_i.
\end{aligned}
\label{eqsMtcHnPlWq}
\end{equation}
Forming the join of these inequalities for all meaningful pairs $(u_1,u_2)$, we obtain that $m_j\leq h_i$, as required.

Second, assume that \eqref{eqpbxdPSWrmXvQpkjsbtz} fails. Then  \wcond{} applied to the (side terms of the)  inequality $p^{(0)}_j \wedge p^{(1)}_j\leq m_j \leq h_1 \vee h_2$ gives that $p^{(0)}_j \leq  h_1 \vee h_2$ or $p^{(1)}_j \leq h_1 \vee h_2$; we can assume that $p^{(0)}_j \leq  h_1 \vee h_2$. By \eqref{eqalignedXSRpPtW} and Lemma~\ref{lemmapjincreasing}, we have that $x_0\leq h_1 \vee h_2$. Since $x_0$ is join prime by Lemma~\ref{lemmagnJMggDprm}, $x_0\leq h_i$ holds for some $i\in\set{1,2}$. Applying \eqref{eqpbxfhszRtdzBg}, we obtain that $h_i=1_{\Sl(n)}$. Hence, $m_j\leq 1_{\Sl(n)} = h_i$, as required. Now that both the validity and the failure of \eqref{eqpbxdPSWrmXvQpkjsbtz} have been considered, 
we conclude 
that $m_j$ is a join prime element of $\Sl(n)$.

Next, we are going to show that $m_j$ is meet prime in $\Sl(n)$. Suppose the contrary, and pick $h_1,h_2\in\Sl(n)$ such that $h_1 \wedge h_2\leq m_j$ but $h_1\nleq m_j$ and $h_2\nleq m_j$. 
We are going to obtain a contradiction by infinite descent.
In order to do so, it suffices to show that the condition 
\begin{equation}
\text{$j-r\in\nfromz$ and there is a $u_r\in[n)$ such that
$h_1 \wedge h_2\leq p^{(u_r)}_{j-r}$}
\label{eqtxtntLnlSzkrSgynP}
\end{equation}
holds for $r=0$, and to show that
\begin{equation}
\parbox{5.4cm}{for every $r\in\nfromz$, if \eqref{eqtxtntLnlSzkrSgynP} holds for $r$, then it also holds for $r+1$.}
\label{eqpbxmSdjRarGszktLnlrSg}
\end{equation}
Applying \wcond{}  to $h_1 \wedge h_2\leq m_j=\bigvee_{u_0<v_0}(p_j^{(u_0)} \wedge p_j^{(v_0)})$, we obtain $u_0,v_0\in[n)$ such that $h_1 \wedge h_2\leq p_j^{(u_0)} \wedge p_j^{(v_0)} \leq p_j^{(u_0)}=p_{j-0}^{(u_0)}$. Hence,  \eqref{eqtxtntLnlSzkrSgynP} holds for $r=0$. 
Next, in order to show \eqref{eqpbxmSdjRarGszktLnlrSg}, assume that $r\in \nfromz$ satisfies condition  \eqref{eqtxtntLnlSzkrSgynP}. We cannot have that $h_1 \wedge h_2\leq x_{u_r}$, because otherwise $h_i\leq x_{u_r}$ for some $i\in\set{1,2}$ by the meet primeness of $x_{u_r}$, see Lemma~\ref{lemmagnJMggDprm}, and so  \eqref{lemmaWcnDtLgnB} would give that $h_i = 0_{\Sl(n)}\leq m_j$, contradicting our assumption. Combining $h_1 \wedge h_2\nleq x_{u_r}$ with \eqref{eqtxtntLnlSzkrSgynP} and \eqref{eqalignedXSRpPtW}, we obtain that $j-r\neq 0$. Hence, $j-(r+1)\in\nfromz$ and  
\begin{equation}
h_1 \wedge h_2 \leq p_{j-r}^{(u_r)} = x_{u_r}\vee\bigvee_{u_{r+1}<v_{r+1}} \bigl( p_{j-(r+1)}^{(u_{r+1})} \wedge  p_{j-(r+1)}^{(v_{r+1})}\bigr).
\label{eqhbLnMhgThtwppl}
\end{equation}
In order to exclude that $h_1\leq p_{j-r}^{(u_r)}$, suppose the contrary. Then, by Lemma~\ref{lemmapjincreasing},  $h_1\leq p_{j-r}^{(u_r)}\leq p_{j}^{(u_r)}$. Using a permutation of $[n)$ with $u_r\mapsto  v_r\in [n)\setminus\set{u_r}$ as in \eqref{eqsMtcHnPlWq},  we  obtain that $h_1 \leq p_{j}^{(v_r)}$. Hence, $h_1\leq p_{j}^{(u_r)}  \wedge p_{j}^{(v_r)}\leq m_j$  is a contradiction, proving that $h_1\nleq p_{j-r}^{(u_r)}$. Similarly, $h_2\nleq p_{j-r}^{(u_r)}$. 
Now that we have seen that $h_1 \wedge h_2 \nleq  x_{u_r}$, $h_1\nleq p_{j-r}^{(u_r)}$, and  $h_2\nleq p_{j-r}^{(u_r)}$, we are in the position to apply \wcond{}  to 
\eqref{eqhbLnMhgThtwppl}. So we obtain  that  $h_1 \wedge h_2 \leq  p_{j-(r+1)}^{(u_{r+1})} \wedge  p_{j-(r+1)}^{(v_{r+1})}$ for some $u_{r+1}, v_{r+1}\in [n)$. 
Hence,  $h_1 \wedge h_2 \leq  p_{j-(r+1)}^{(u_{r+1})}$ and \eqref{eqtxtntLnlSzkrSgynP} holds for $r+1$. We have verified \eqref{eqpbxmSdjRarGszktLnlrSg}. Thus, it follows that \eqref{eqtxtntLnlSzkrSgynP} holds for all $r\in\nfromz$. This
is the required contradiction proving that $m_j$ is meet prime in $\Sl(n)$, completing  the proof of Lemma~\ref{lemmamjisdoubleprime}.
\end{proof}

For $3\leq n\in\nfromo$,  in connection with \eqref{eqalignedXSRpPtW} and Lemma~\ref{lemmabottopNu}, let
\begin{equation*} 
P_n:=\begin{cases}
\set{m_j\in \Sl(n):j\in\nfromo}\cup \set{\ovl{m_j}\in \Sl(n):j\in\nfromo},&\text{ if }n=3\cr
\set{m_j\in \Sl(n):j\in\nfromz}\cup \set{\ovl{m_j}\in \Sl(n):j\in\nfromz},&\text{ if }n>3.
\end{cases}
\end{equation*} 
With the ordering of $\Sl(n)$ restricted to $P_n$, $\alg P_n=(P_n;\leq)$ is an ordered set, which is described by the following lemma; see also Figure~\ref{figConstr} for $3<n\in\nfromo$. This lemma explains why the case $n=3$ differs slightly from the case $n>3$.

\begin{lemma} 
\label{lemmam1notsubdual} The following four assertions hold.
\begin{enumeratei}
\item\label{lemmam1notsubduala} If $n>3$, then $s<m_0 \nleq \ovl{m_0}$.
\item\label{lemmam1notsubdualb} If $n=3$, then $m_1\not\leq\ovl{m_1}$. However, $m_0=s<\ovl s=\ovl{m_0}$.
\item\label{lemmam1notsubdualc} For  $n\geq 3$ and $i,j\in\nfromo$, $m_i\nleq \ovl{m_j}$.  
\item\label{lemmam1notsubduald} For  $n\geq 3$ and $i,j \in\nfromz$, $\ovl{m_j}\nleq m_i$.  
\end{enumeratei}
\end{lemma}

\begin{proof} 
If $n>3$, then letting $\vec w:=(1,1,0,\dots,0)$, we have that $s(\vec w)=0$,  $m_0(\vec w)=1$,  and $\ovl{m_0}(\vec w)=0$ hold in the two-element lattice $\mathbf{2}$.
This proves \eqref{lemmam1notsubduala}, because $s\leq m_0$ follows from Lemma~\ref{lemmabottopNu}.  

Next, to deal with \eqref{lemmam1notsubdualb}, we assume that $n=3$. Clearly, $m_0=s$ and $\ovl{m_0}=\ovl s$. Lemma~\ref{lemmabottopNu} gives that $s\leq\ovl s$ while $\sf M_3$ witnesses that $s\neq\ovl s$. We have seen that $m_0=s<\ovl s=\ovl{m_0}$. 
For the sake of contradiction, suppose that $m_1\leq \ovl{m_1}$. Then each of the joinands of $m_1$ is less than or equal to every meetand of $\ovl{m_1}$. In particular, we have that $p_1^{(0)} \wedge p_1^{(1)}$ is less than or equal to its dual, that is, 
\begin{equation}
\begin{aligned}
&\bigl(x_0\vee (x_1 \wedge x_2)\bigr) \wedge  
\underline{\bigl(x_1\vee (x_0 \wedge x_2)\bigr)} 
\cr
&\leq \bigl(x_0 \wedge (x_1\vee x_2)\bigr) \vee \bigl( x_1 \wedge (x_0\vee x_2)\bigr).
\end{aligned}
\label{eqbrCdpRtkpcs}
\end{equation}
By \wcond{}, duality, and since $x_0$ and $x_1$ play symmetric roles, we can assume that \eqref{eqbrCdpRtkpcs} holds after omitting its underlined meetand. 
Hence, $x_0$ is less than or equal to the right hand side of 
 \eqref{eqbrCdpRtkpcs}. Using that $x_0$ is join prime by Lemma~\ref{lemmagnJMggDprm}, we obtain that either 
$x_0\leq x_0 \wedge (x_1\vee x_2)\leq x_1\vee x_2$, or $x_0\leq x_1 \wedge (x_0\vee x_2)\leq x_1$, so we have obtained a contradiction, proving \eqref{lemmam1notsubdualb}.

Next, we turn our attention to \eqref{lemmam1notsubdualc}. Suppose, for a contradiction, that $m_i\leq \ovl{m_j}$ for some $i,j\in\nfromo$. We obtain from  Lemma~\ref{lemmamjisincreasing}  that 
\[m_0<m_1\leq m_i\leq \ovl{m_j}\leq \ovl{m_1}<\ovl{m_0}.
\] 
In particular, $m_0<\ovl{m_0}$ and $m_1\leq\ovl{m_1}$. The first
of these two inequalities contradicts part \eqref{lemmam1notsubduala} if $n>3$, while the second one contradicts part \eqref{lemmam1notsubdualb} if $n=3$. So we obtain a contradiction for all $n\geq 3$, whereby   \eqref{lemmam1notsubdualc} holds.

Finally, we obtain \eqref{lemmam1notsubduald} basically from $\eta(\ovl{m_j})=1$ and $\eta(m_i)=0$, see \eqref{eqnatHmTdvsz}, as follows.  Using \eqref{eqtxtdhbmzTwMn}, its dual,  and   \eqref{eqalignedXSRpPtW} defining our terms, we obtain that $\ovl{m_j}(a_0,\dots, a_{n-1})=1_{\ssf M_n}$ and 
 $m_i(a_0,\dots, a_{n-1})=0_{\ssf M_n}$.  This implies  \eqref{lemmam1notsubduald} and completes the proof of Lemma~\ref{lemmam1notsubdual}.
\end{proof}

Now, to indicate that we are progressing in the desired direction, 
we are going to formulate a corollary. Note, however, that neither this corollary, nor its proof, nor the concept defined in the present paragraph will be used in this paper, so the reader can skip over them.
Following Dean~\cite{deanCF} and Dilworth~\cite{dilworthLUB},  an ordered set  $P=(P;\leq_P)$ \emph{completely freely generates} a lattice $K$ if $P$ is a subset of $K$, $\leq_P$ is the restriction of the lattice order $\leq_K$ to $P$, and for every lattice $L$ and every order-preserving map $\phi\colon (P;\leq_P)\to L$, there exists a lattice homomorphism $K\to L$ that extends $\phi$. If so, then we denote $K$ by $\cfree(P;\leq)$. 
 The ordered set $(P_n;\leq)$ was defined right before (and in) Lemma~\ref{lemmam1notsubdual}; 
let $[P_n]_{\fl n}=[P_n]_{\Sl(n)}$ denote the sublattice generated by it.

\begin{corollary}\label{corolddTPsT}\
\begin{enumeratei}
\item\label{corolddTPsTa}
 As an ordered set, $(P_n;\leq)$ is described by Lemmas~\ref{lemmamjisincreasing} and \ref{lemmam1notsubdual}; note that  for $3<n\in\nfromo$, $(P_n;\leq)$ is given also by Figure~\ref{figConstr}. 
\item\label{corolddTPsTb} The sublattice $[P_n]_{\fl n}=[P_n]_{\Sl(n)}$ is completely freely generated by $(P_n;\leq)$. 
\item\label{corolddTPsTc}
Furthermore, $\Sl(n)$ has a sublattice isomorphic to $\fl\omega$.
\end{enumeratei}
\end{corollary}

Part~\eqref{corolddTPsTc} is a consequence of Theorem~\ref{thmmain}\eqref{thmmainb}; the point is that we can easily conclude Corollary~\ref{corolddTPsT}\eqref{corolddTPsTc} from known results and the previous lemmas.

\begin{proof}[Proof of  Corollary~\ref{corolddTPsT}]
As opposed to other proofs in the paper, the present argument relies on some outer references that are not quoted with full details. Part~\eqref{corolddTPsTa} is clear. For the validity of Part~\eqref{corolddTPsTb}, we need to show that for arbitrary $(k+k)$-ary  lattice terms $t_1$ and $t_2$, the inequality 
\begin{equation}
t_1(m_i:i<k,\, \ovl{m_i}: i<k)\leq t_2(m_i:i<k,\, \ovl{m_i}: i<k)
\label{eqdNszBmfTsztLntk}
\end{equation}
holds in $\fl n$ iff it holds in the completely free lattice $\cfree(P_n;\leq)$. 
The satisfaction of \eqref{eqdNszBmfTsztLntk} in $\cfree(P_n;\leq)$ can be tested by Dean's algorithm, which is a generalization of Whitman's algorithm; see Dean~\cite{deanCF} or see  Freese, Je\v zek, and Nation~\cite[Theorem 5.19]{fjnmonograph}. This is a recursive algorithm that uses only the following three properties of $\cfree(P_n; \leq)$ and $(P_n; \leq)$: 
\begin{enumerate}[(\upshape{D}1)\quad]
{\setlength\itemindent{\aindent}\item\label{enumDalga} $\cfree(P_n ; \leq)$ satisfies \wcond{},} 
{\setlength\itemindent{\aindent}\item\label{enumDalgb} the elements of $P_n$ are doubly prime, and}
{\setlength\itemindent{\aindent}\item\label{enumDalgc} the description of the ordering of $P_n$.} 
\end{enumerate}
It follows from Lemmas~\ref{lemmaWcnDtLgnB},  \ref{lemmamjisincreasing},  \ref{lemmamjisdoubleprime}, and \ref{lemmam1notsubdual} that these properties hold for $P_n$ as a subset of $\fl n$. Therefore, Dean's algorithm gives the same result in $\cfree(P_n;\leq)$ as (D\ref{enumDalga})--(D\ref{enumDalgc}) give
in $\fl n$. Hence, we conclude that  $[P_n]_{\fl n}$ is completely freely generated by $(P_n;\leq)$, proving Part~\eqref{corolddTPsTb}.
Thus, $\cfree(P_n;\leq)$ can be embedded into $\Sl(n)$. Using that $\fl\omega$ can be embedded into $\fl 3$ by Whitman~\cite{whitman} and $\fl 3$ can be embedded into $\cfree(P_n;\leq)$ by the main result of Rival and Wille~\cite{rivalwille}, we conclude by transitivity that $\Sl(n)$ has a sublattice isomorphic to $\fl\omega$. This proves Part~\eqref{corolddTPsTc}.
\end{proof}

Now, we are ready to prove our Key Lemma.

\begin{proof}[Proof of Lemma \ref{lemmakeylemma}]
It is clear by \eqref{eqalignedXSRpPtW} that
\begin{equation}
\set{m_j:j\in\nfromz}\cup\set{\ovl{m_j}:j\in\nfromz}\cup\set{a,b,\ovl a,\ovl b} \subseteq \Sl(n).
\label{eqsznbhrTsstQ}
\end{equation}
In order to apply Lemma~\ref{lemmafjnFrGSSt} and complete the proof in this way, it suffices to show that none of the  inequalities 
\begin{enumerate}[({\upshape ineq}1)\quad]
{\setlength\itemindent{\bindent}\item\label{afe1} $x_0\leq a \vee \ovl{a} \vee b \vee \ovl{b}$,}
{\setlength\itemindent{\bindent}\item\label{afe2} $a\leq x_0 \vee \ovl{a} \vee b \vee \ovl{b}$,}
{\setlength\itemindent{\bindent}\item\label{afe3} $b\leq x_0 \vee a \vee \ovl{a} \vee \ovl{b}$,}
{\setlength\itemindent{\bindent}\item\label{afe4} $\ovl{a}\leq x_0 \vee a \vee b \vee \ovl{b}$, and}
{\setlength\itemindent{\bindent}\item\label{afe5} $\ovl{b}\leq x_0 \vee a \vee \ovl{a} \vee b$}
\end{enumerate}
holds in $\fl n$, because then the same will be true for their duals. For example, if  
$a\geq x_0  \wedge  \ovl{a} \wedge  b \wedge   \ovl{b}$ held, then we could apply  $\delta=\delta_{\fl n}$ from \eqref{eqpbxnatduAl} to this inequality to obtain  that \iref{afe4} holds.

First, we consider  \iref{afe1}. Suppose, for a contradiction, that it holds. Using \eqref{eqsznbhrTsstQ}, the elements we are going to deal are in $\Sl(n)$. For $i\in\nfromo$, we have that $m_i<m_{i+1}$ by Lemma~\ref{lemmamjisincreasing}, whereby Lemma~\ref{lemmabottopNu} gives that $m_i\leq \ovl s$. Since $m_0 < m_i$ by Lemma~\ref{lemmamjisincreasing}, Lemma~\ref{lemmabottopNu} gives that $s\leq m_i$, whence $\ovl{m_i}\leq \ovl s$. Since $m_i\leq \ovl s$ and $\ovl{m_i}\leq \ovl s$ for all $i\in\nfromo$,  \eqref{eqalignedXSRpPtW} gives that $a \vee \ovl{a} \vee b \vee \ovl{b}\leq \ovl s$. This is a contradiction, because \eqref{eqpbxfhszRtdzBg}
and \iref{afe1} imply that $a \vee \ovl{a} \vee b \vee \ovl{b} = 1_{\Sl(n)} > \ovl s$.

Second, for the sake of contradiction, suppose that  \iref{afe2} holds. Since $\ovl{m_3}\leq a$, we obtain that 
\begin{equation}
\ovl{m_3}\leq x_0 \vee (\ovl{m_1} \wedge m_3) \vee m_2 \vee \ovl{m_4} \vee (\ovl{m_2} \wedge m_4). 
\label{eqdhRtRspDkTnvR}
\end{equation}
By Lemma~\ref{lemmamjisincreasing}, $\ovl{m_3}\nleq \ovl{m_4}$. None of the inequalities 
$\ovl{m_3}\leq \ovl{m_1} \wedge m_3\leq m_3$, $\ovl{m_3}\leq m_2$, and $\ovl{m_3}\leq \ovl{m_2} \wedge m_4\leq m_4$ holds by Lemma~\ref{lemmam1notsubdual}\eqref{lemmam1notsubduald}. 
Thus, since $\ovl{m_3}$ is join prime by Lemma~\ref{lemmamjisdoubleprime}, it follows from \eqref{eqdhRtRspDkTnvR} that $\ovl{m_3}\leq x_0$. Hence, 
 \eqref{eqpbxfhszRtdzBg} yields that $\ovl{m_3}=0_{\Sl(n)}$. So 
 $0_{\Sl(n)} = \ovl{m_3}>\ovl{m_4} \in \Sl(n)$ by Lemma~\ref{lemmamjisincreasing}, and this is a contradiction.
Thus, \iref{afe2} fails, as required.

Third, for the sake of  contradiction, we suppose that  \iref{afe3} holds. Since $m_2\leq b$, we obtain that 
\begin{equation}
m_2\leq  x_0 \vee   m_1 \vee \ovl{m_3}   \vee   (\ovl{m_1} \wedge m_3)   \vee  (\ovl{m_2} \wedge m_4). 
\label{eqdpRtRssvwhzW}
\end{equation}
By Lemma~\ref{lemmamjisincreasing},  $m_2\nleq m_1$. None of the inequalities 
$m_2\leq \ovl{m_1} \wedge m_3 \leq \ovl{m_1}$, $m_2\leq \ovl{m_3}$, 
and $m_2\leq \ovl{m_2}\wedge m_4\leq \ovl{m_2}$ holds by Lemma~\ref{lemmam1notsubdual}\eqref{lemmam1notsubdualc}. 
Therefore, since $m_2$ is join prime by Lemma~\ref{lemmamjisdoubleprime}, \eqref{eqdpRtRssvwhzW} gives that $m_2\leq x_0$. Hence, \eqref{eqpbxfhszRtdzBg} and Lemma~\ref{lemmamjisincreasing} yield that 
$m_1<m_2=0_{\Sl(n)}$, contradicting $m_1\in\Sl(n)$. 
Therefore, \iref{afe3} fails, as required.

Clearly, there is a lot of similarity between the treatment for \iref{afe2} and that for \iref{afe3}. Namely, both arguments rely on  \eqref{eqpbxfhszRtdzBg}, Lemmas~\ref{lemmamjisincreasing}, \ref{lemmamjisdoubleprime}, and \ref{lemmam1notsubdual}, and 
some comparabilities among the subscripts. In an analogous way, the argument for \iref{afe4} and that for \iref{afe5} are also very similar; this justifies that  only the first of them will be detailed.

For the sake of  contradiction, suppose that \iref{afe4}  is satisfied, that is,  
\begin{equation}
\ovl{m_1} \wedge m_3 \leq x_0 \vee m_1 \vee \ovl{m_3} \vee  m_2 \vee \ovl{m_4} \vee \underline{(\ovl{m_2} \wedge m_4)}.
\label{eqdzbnTpQrgNmGlVmh}
\end{equation}
We are going to parse this inequality by \wcond{}, taking into account that,  according to  Lemmas \ref{lemmagnJMggDprm} and \ref{lemmamjisdoubleprime}, both meetands on the left and the five non-underlined joinands on the right of   \eqref{eqdzbnTpQrgNmGlVmh} are doubly prime elements. Therefore, either one of the two meetands is less than or equal to one of the five non-underlined joinands, or $\ovl{m_1} \wedge m_3 \leq \ovl{m_2} \wedge m_4$; so we need to consider only these possibilities.  If we had that $\ovl{m_1}\leq x_0$ or $m_3\leq x_0$, then  \eqref{eqpbxfhszRtdzBg} and Lemma~\ref{lemmamjisincreasing} would lead to $\ovl{m_2}<\ovl{m_1}=0_{\Sl(n)}$ or $m_2<m_3=0_{\Sl(n)}$, which are contradictions.
If one of the  two meetands was less than or equal to another 
non-underlined joinand, then Lemma~\ref{lemmamjisincreasing} or Lemma~\ref{lemmam1notsubdual} would prompt give a contradiction. We are left with the case $\ovl{m_1} \wedge m_3 \leq \ovl{m_2} \wedge m_4$, but then  $\ovl{m_1} \wedge m_3 \leq \ovl{m_2}$, so the meet primeness of $\ovl{m_2}$ gives that $\ovl{m_1} \leq \ovl{m_2}$, contradicting Lemma~ \ref{lemmamjisincreasing}, or $m_3 \leq \ovl{m_2}$, contradicting Lemma~\ref{lemmam1notsubdual}\eqref{lemmam1notsubdualc}. Therefore, \iref{afe4} fails, as required. Finally, as we have already mentioned,  \iref{afe5} fails by an analogous argument.
This completes the proof of Lemma~\ref{lemmakeylemma}
\end{proof}

\section{From the Key Lemma to a stronger statement}\label{sectionfromKeyL}
If a subset $X$ of a lattice freely generates, then so do the subsets of $X$. Thus, (the Key) Lemma~\ref{lemmakeylemma} in itself implies  that, for every $3\leq \lambda\in\nfromo$, there is a totally symmetric embedding $\fl4\to\fl\lambda$. In particular, Lemma~\ref{lemmakeylemma}  implies 
Corollary~\ref{corolnhRm}. 
In this section, with the extensive help of Cz\'edli~\cite{czg-selfdual}, we lift the rank $4$ of $\fl 4$ to all even natural numbers $\kappa\geq 4$ and even to $\aleph_0$.  That is, we are going to prove the following lemma. Remember that $a$, $\ovl a$, $b$ and $\ovl b$ have been defined in \eqref{eqalignedXSRpPtW}.

\begin{lemma}\label{lemmawcMsKlMm} If $\kappa=\aleph_0$ or $\kappa\geq 4$ is an even integer, then for every integer $\lambda=n\geq 3$,  there exists a totally symmetric embedding $\tau_{\kappa\lambda}\colon\fl\kappa\to\fl\lambda$ with the additional property that $\tau_{\kappa\lambda}(\fl\kappa)\subseteq [a,\ovl a,b,\ovl b]$. 
\end{lemma}

\begin{proof}
First, in order to make our references to Cz\'edli~\cite{czg-selfdual} convenient, we need to deal with the notation. Let $(y_1,y_2,y_3,y_4):=(a,\ovl a,b,\ovl b)\in \Sl(n)^4$. It follows from (the Key) Lemma~\ref{lemmakeylemma} that $\set{y_1,\dots, y_4}$ freely generates. This allows us to write $\fl 4=\Fl(y_1,\dots,y_4)$ in the present proof, so $\fl 4$  is a sublattice of $\Sl(n)$. 
Since $\set{y_1,\dots,y_4}=\set{a,\ovl a,b,\ovl b}$ is closed with respect to $\delta_{\fl n}$ defined in \eqref{eqpbxnatduAl},  $\fl 4=\Fl(y_1,\dots,y_4)$ is selfdually positioned in $\fl n$. 
Hence, 
\begin{equation}
\text{the restriction $\dswap4:=\restrict{\delta_{\fl n}}{\fl 4}$  of $\delta_{\fl n}$ to $\fl 4$,}
\label{eqtxtdfrrStrdmgSzRmnG}
\end{equation}
is a dual automorphism of $\fl 4$. Note the rule that
\begin{equation}
\dswap4(y_1)=y_2,\quad \dswap4(y_2)=y_1,\quad\dswap4(y_3)=y_4,\quad\dswap4(y_4)=y_3.
\label{eqBsthRlmTlBlDgL}
\end{equation}
We do not need the exact definition of the lattice terms
$x_1^{1+i}$ and $x_2^{1+i}$ given in  \cite[Section 4]{czg-selfdual}, but we have to recall some of their properties. For $i\in\nfromz$, 
$x_1^{1+i}$ and $x_2^{1+i}$ are lattice terms over $\set{y_1,\dots,y_4}$, that is, they belong to $\fl 4$. For brevity, we  denote $\set{x_1^{1+i}: 2i<\kappa}\cup \set{x_2^{1+i}: 2i<\kappa}$  by $\set{x_1^{1+i}, x_2^{1+i}: 2i<\kappa}$. By \cite[Lemma 4.1]{czg-selfdual}, $\set{x_1^{1+i}, x_2^{1+i}: 2i<\kappa}$ freely generates a sublattice of $\fl 4$. Hence, in the present proof, we can write that 
\[\fl\kappa=\Fl(x_1^{1+i}, x_2^{1+i}: 2i<\kappa)=[x_1^{1+i}, x_2^{1+i}: 2i<\kappa]_{\Fl(4) }.
\]
For example, $\fl 6=\Fl(x_1^{1+i}, x_2^{1+i}: 2i<6)=\Fl(x_1^1,x_2^1, x_1^2,x_2^2, x_1^3,x_2^3)$. It is important that 6 and, in general, $\kappa$ is even or $\aleph_0$, because for an odd integer $\kappa\in\nfromo$, 
$\Fl(x_1^{1+i}, x_2^{1+i}: 2i<\kappa)$ would be $\Fl(\kappa+1)$ rather than $\fl\kappa$.

This paragraph is to tailor the second half of \cite[Lemma 4.1]{czg-selfdual} to the present situation;  the reader may want to skip over it. 
It is irrelevant for us what $a$ and $b$ denote in  \cite{czg-selfdual}; they are not the same as here. It is also irrelevant that $\fl 4=\Fl(y_1,\dots, y_4)$ is embedded into $\Fl(x,y,z)$ in \cite{czg-selfdual} but 
into $\fl n$ here; these two embeddings are different even for $n=3$.  
Using the first page, Lemma 2.1(B), and the second line of Section 4 of  \cite{czg-selfdual},  we obtain from  \cite{czg-selfdual}  that $\delta=\restrict{\delta_{\Fl(x,y,z)}}{\Fl(y_1,\dots,y_4)}$ in \cite[Lemma 4.1]{czg-selfdual} denotes a dual automorphism of 
$\fl4=\Fl(y_1,\dots,y_4)$ such that \eqref{eqBsthRlmTlBlDgL} holds also for $\delta$. Therefore, $\delta$ in  \cite[Lemma 4.1]{czg-selfdual} is the same as 
$\dswap4$ here, and the second half of  \cite[Lemma 4.1]{czg-selfdual} asserts that 
$\dswap4(x_1^i)=x_2^i$ and $\dswap4(x_2^i)=x_1^i$ for all meaningful $i$. 

So,  \cite[Lemma 4.1]{czg-selfdual} yields that 
$\dswap4(x_1^i)=x_2^i$ and $\dswap4(x_2^i)=x_1^i$ hold for all $i$ such that $2i<\kappa$. Therefore,
since $\dswap4$ is the restriction of $\delta_{\fl n}$ by \eqref{eqtxtdfrrStrdmgSzRmnG}, the set $\set{(x_1^{1+i}, x_2^{1+i}: 2i<\kappa}$ is selfdually positioned in $\fl n$. Consequently, so is the sublattice $\fl\kappa=[x_1^{1+i}, x_2^{1+i}: 2i<\kappa]_{\fl n}$. 
Finally, $\fl\kappa\subseteq [y_1,\dots, y_4]_{\fl n}=[a,\ovl a,b,\ovl b]_{\fl n}\subseteq\Sl(n)$. Hence, the inclusion map $\tau_{\kappa\lambda}\colon\fl\kappa\to \fl n=\fl\lambda$ is a totally symmetric embedding and  $\tau_{\kappa\lambda}(\fl\kappa)\subseteq [a,\ovl a,b,\ovl b]$.
This completes the proof of Lemma~\ref{lemmawcMsKlMm}.
\end{proof}

\section{The rest of the proofs}\label{sectionRestPrf}
\begin{proof}[Proof of Theorem \ref{thmSD}]
If $\fl\lambda$ has a selfdually positioned sublattice isomorphic to $\fl\kappa$, then 
$\max\set{\kappa,\aleph_0} =|\fl\kappa|\leq |\fl\lambda|=
\max\set{\lambda,\aleph_0}$,
as required. Conversely, assume that $\max\set{\kappa,\aleph_0}\leq \max\set{\lambda,\aleph_0}$. To specify the free generators, we let 
$\fl\kappa=\Fl(y_i:i < \kappa)$ and $\fl\lambda=\Fl(x_i:i<\lambda)$. 
We can assume that $\kappa>\lambda$, since otherwise the sublattice
$[x_i:i<\kappa]$ generated by $\set{x_i:i<\kappa}$ in $\fl\lambda$ is selfdually positioned and it is isomorphic to $\fl\kappa$. 
The inequality $\kappa>\lambda$ together with $\max\set{\kappa,\aleph_0}\leq \max\set{\lambda,\aleph_0}$ and $3\leq\kappa$ give that $4\leq \kappa\leq \aleph_0$ and $\lambda\in\nfromo$.
We can assume that $\kappa=2k+1\geq 4$ is an odd integer, since otherwise Lemma~\ref{lemmawcMsKlMm} gives a totally symmetric embedding $\phi\colon\fl\kappa\to\fl\lambda$ and $\phi(\fl\kappa)$ does the job.  Again by Lemma~\ref{lemmawcMsKlMm}, we can pick a  $2k$-element subset $C=\set{c_i:i<2k}$ of the sublattice $[a,\ovl a, b,\ovl b]$ of $\fl\lambda$ such that 
the sublattice $[C]$ is selfdually positioned in $\fl\lambda$ and $[C]$ is freely generated by $C=\set{c_i:i<2k}$.  
Since the natural dual automorphism $\delta_{\fl\lambda}$ from \eqref{eqpbxnatduAl} preserves double irreducibility in $[C]$, it follows, after slight notational changes, from
\eqref{aligntxtDfkrZntNshG} and \eqref{eqtxtFpNslmnTwszpThG} that the set $C$ itself is selfdually positioned, that is, $\delta_{\fl\lambda}(C)=C$. Let $D=C\cup\set{x_0}$. Since $\delta_{\fl\lambda}(x_0)=x_0$ by definition, $\delta_{\fl\lambda}(D)=D$. This yields that $[D]=[D]_{\fl\lambda}$, the sublattice generated by $D$, is selfdually positioned, that is, 
$\delta_{\fl\lambda}([D])=[D]$. Therefore, since $|D|=2k+1=\kappa$, 
we need to show only that $D$ freely generates. By duality and Lemma~\ref{lemmafjnFrGSSt}, it suffices to exclude that
\begin{align}
x_0&\leq c_0 \vee c_1 \vee \dots \vee c_{2k-1}, \quad \text{ or }\label{alignXclkCdtGa}\\
c_j&\leq x_0 \vee \bigvee _{i\in [2k)\setminus\set j}c_i\quad
\text{ for some } j\in[2k).\label{alignXclkCdtGb}
\end{align}
We know from Lemma~\ref{lemmakeylemma} that the sublattice $S:=[a,\ovl a, b,\ovl b,x_0]$ of $\fl n$ is freely generated by the set $\set{a,\ovl a, b,\ovl b,x_0}$. Therefore, the self-maps
\[
\xi_1:=\begin{pmatrix}
a& \ovl a& b & \ovl b & x_0\cr
0_S& 0_S& 0_S & 0_S & 1_S
\end{pmatrix}
\quad\text{and}\quad
\xi_2:=\begin{pmatrix}
a& \ovl a& b & \ovl b & x_0\cr
a& \ovl a& b & \ovl b &0_S
\end{pmatrix}
\]
extend to endomorphisms $\what\xi_1\colon S\to S$ and 
$\what\xi_2\colon S\to S$, respectively. 
Using the inclusion $C\subseteq [a,\ovl a, b,\ovl b]$, we obtain that  $\xi_1([a,\ovl a, b,\ovl b])=\set{0_S}$. Taking the equality
$\xi_1(x_0)=1_S$ also into account,  we obtain that the endomorphism $\what \xi_1$ does not preserve inequality \eqref{alignXclkCdtGa}. Hence, \eqref{alignXclkCdtGa} fails, as required. Next, suppose that  \eqref{alignXclkCdtGb} holds. 
Using $C\subseteq [a,\ovl a, b,\ovl b]$, it follows that the restriction of  $\what\xi_2$
to $[a,\ovl a, b,\ovl b]$ is the identity map. Therefore, since $\what\xi_2$ is order-preserving, its application to  \eqref{alignXclkCdtGb} yields that $c_j\leq \bigvee_{i\in [2k)\setminus\set j}c_i$. But this is a contradiction since $C$ freely generates, and we conclude that  \eqref{alignXclkCdtGb} fails, as required. 
\end{proof}

\begin{proof}[Proof of Corollary~\ref{corolDA}]
In order to prove the implication  \eqref{corolDAb} $\Rightarrow$ \eqref{corolDAa}, 
assume that \eqref{corolDAb} holds. We can also assume that $\kappa\neq\lambda$ since otherwise the identity map  $\fl\kappa\to\fl\kappa=\fl\lambda$ does the job.  As it is pointed out right after \eqref{eqpbxdnwlmghTrzzdtnbQ},  $[\Daut(\fl\lambda):\Aut(\fl\lambda)]=2$. Hence, with  $\delta_{\fl\lambda}$ from \eqref{eqpbxnatduAl},
\begin{equation}
\Daut(\fl\lambda)=\Aut(\fl\lambda)\cup\set{\delta_{\fl\lambda}\circ\phi: \phi\in \Aut(\fl\lambda)}.
\label{eqkndXsTWsktt}
\end{equation}
Thus, the embedding given by  Lemma~\ref{lemmawcMsKlMm} has a \DAstr-symmetric range. Hence, \eqref{corolDAb} $\Rightarrow$ \eqref{corolDAa}.

Before proving the converse implications, we formulate and verify some observations, some of which will be useful also in the proof of Corollary~\ref{corolaUtsym} later. This is why instead of assuming \DAstr-symmetry, we often assume less, the usual symmetry (with respect to automorphisms).
Since $\Aut(\fl X)$ acts transitively on the set $X$ of free generators, it follows  trivially that
\begin{equation}
\parbox{6.2cm}{if $S$ is a symmetric sublattice of $\Fl(X)$ such that $S\cap X\neq\emptyset$, then $S=\fl\lambda$.}
\label{eqpbxXhbBnKrCsdJh}
\end{equation}
As a straightforward consequence of \eqref{aligntxtDfkrZntNshG}, observe that 
\begin{equation}
\parbox{7.5cm}{if $\phi\colon \fl\kappa\to\fl\lambda$ is an arbitrary embedding and $S:=\phi(\fl\kappa)$, then $|\Dirr S|=\kappa$.}
\label{eqpbxhmhzlRfgldzlvs}
\end{equation}
Since automorphisms and dual automorphisms preserve double primeness, we obtain the following observation.
\begin{equation}
\parbox{9cm}{Let $S$ be a symmetric sublattice of $\fl\lambda$; then $\restrict\tau{\Dirr S}\colon \Dirr S\to\Dirr S$ is a bijective map for every $\tau\in \Aut(\fl\lambda)$. If, in addition, $S$ is \DAstr-symmetric, then the same holds even for every $\tau\in\Daut(\fl\lambda)$.} \label{eqtxtlpFplmsSnVlTNm}
\end{equation}

For $u\in\fl\lambda$, the set $\set{\tau(u): \tau\in\Aut(\fl\lambda)}$ will be called the \emph{orbit} of $u$ (with respect to automorphisms). We are going to prove the following property of orbits. 
\begin{equation}
\parbox{8.6cm}{If $\lambda\geq \aleph_0$, $S$ is a symmetric sublattice of $\fl\lambda$, and $u\in S$, then $|S|=\lambda= |\set{\tau(u): \tau\in\Aut(\fl\lambda)}|$.}
\label{pbxdznGhHmQprnhGr}
\end{equation}
In order to show \eqref{pbxdznGhHmQprnhGr}, let $S$ be a symmetric sublattice of $\fl\lambda=\fl X$, where $|X|=\lambda$, and let $u\in S$. Obviously, $|S|\leq |\fl\lambda|=\lambda$. 
 It is clear by \eqref{eqpbxtrTtRlMnts} that there is a finite subset $Y\subseteq X$ such that $u$ is in the sublattice $[Y]$ generated by $Y$. By the rudiments of cardinal arithmetics, there is a family $\set{\pi_i : i<\lambda}$ of permutations of $X$ such that 
$\phi_i(Y)\cap\phi_j(Y)=\emptyset$ for $i\neq j$. Each of these $\pi_i$ extends to an automorphism $\ketparamaut \pi i$ of $\fl\lambda$. If $i\neq j$, then the map 
\[
X\to \alg 2, \quad x\mapsto
\begin{cases}
1,&\text{if }x\in \pi_i(Y),\cr
0,&\text{if }x\notin \pi_i(Y),\text{ in particular, if }x\in\pi_j(Y)\cr
\end{cases}
\]
extends to a lattice homomorphism $\fl\lambda\to\alg 2$. 
Since this homomorphism maps $\ketparamaut \pi i(u)\in [\pi_i(Y)]$ and 
 $\ketparamaut \pi j(u)\in [\pi_j(Y)]$ to 1 and 0, respectively, we obtain that   $\ketparamaut \pi i(u)\neq \ketparamaut \pi j(u)$. Furthermore, $\set{\tau(
u): \tau\in\Aut(\fl\lambda)}\subseteq S$ since $S$ is a symmetric sublattice. Hence, 
\[\lambda=|\set{\pi_i(u): i<\lambda}|\leq 
|\set{\tau(u):  \tau\in\Aut(\fl\lambda)}| \leq|S|\leq \lambda,
\]
which proves \eqref{pbxdznGhHmQprnhGr}. 
We also need the following consequence of \eqref{pbxdznGhHmQprnhGr}.
\begin{equation}
\parbox{6.6cm}{If $\lambda\geq \aleph_0$, $S$ is a symmetric sublattice of $\fl\lambda$, and $\Dirr S\neq \emptyset$, then $|\Dirr S|=\lambda$.}
\label{pbxdznthHmQpfcskmpdLbnDGr}
\end{equation}
In order to show this, let $u\in\Dirr S$. Since $S$ is symmetric, the restriction $\restrict\tau S$ of an automorphism $\tau\in\Aut(\fl\lambda)$ to $S$ is an automorphism of $S$. Hence, $\tau(u)=\restrict\tau S(u)$ also belongs to $\Dirr S$. Thus, we conclude  from \eqref{pbxdznGhHmQprnhGr} that $\lambda\leq |\Dirr S|\leq |\fl\lambda|=\lambda$, implying the validity of \eqref{pbxdznthHmQpfcskmpdLbnDGr}.
In the observation below, $\delta=\delta_{\fl\lambda}$ is the natural dual automorphism introduced in \eqref{eqpbxnatduAl}.
An \emph{involution} on a set $Y$ is a map $Y\to Y$ whose square is the identity map on $Y$. 
\begin{equation}
\parbox{9cm}{If $S$ is a \DAstr-symmetric sublattice of $\fl\lambda$, then the restriction $\restrict\delta{\Dirr S}$
of $\delta$ to $\Dirr S$ is an involution on $\Dirr S$.}
\label{eqpbxhmrKpLkBtNnS}
\end{equation}
 Since every restriction of an involution is again an involution, \eqref{eqpbxhmrKpLkBtNnS} follows immediately from \eqref{eqtxtlpFplmsSnVlTNm}. 
Next, we are going to prove that 
\begin{equation}
\parbox{9cm}{if $\kappa\neq\lambda$ and there is an embedding $\fl\kappa\to\fl\lambda$ with \DAstr-symmetric range, then 
$\kappa$ is not an odd integer.}
\label{eqpbxdzBmztpLsjflS}
\end{equation}
Suppose the contrary,  and  for an odd $\kappa\in\nfromo$, let 
$\phi\colon \fl\kappa\to\fl\lambda$ with range $S:=\phi(\fl\kappa)$ 
witness the failure of \eqref{eqpbxdzBmztpLsjflS}. We know from \eqref{eqpbxhmrKpLkBtNnS}  that $\restrict\delta{\Dirr S}$ is an involution  on $\Dirr S$. If $\restrict\delta{\Dirr S}$ has a fixed point $u\in \Dirr S$, then $u= \restrict\delta{\Dirr S}(u)=\delta(u)$ is one of the free generators of $\fl\lambda$ by \eqref{eqtxtFpNslmnTwszpThG}, whereby \eqref{eqpbxXhbBnKrCsdJh} gives the equality in $\fl\kappa\cong S= \fl\lambda$, which implies $\kappa=\lambda$ by \eqref{aligntxtDfkrZntNshG}, contradicting $\kappa\neq \lambda$. Hence, $\restrict\delta{\Dirr S}$ has no fixed point. By \eqref{eqpbxhmhzlRfgldzlvs}, this fixed-point-free involution acts on a $\kappa$-element set. Thus, $\kappa$ is not an odd integer, proving \eqref{eqpbxdzBmztpLsjflS}.

Now, armed with \eqref{pbxdznGhHmQprnhGr}, \eqref{pbxdznthHmQpfcskmpdLbnDGr}, and \eqref{eqpbxdzBmztpLsjflS},   we are in the position to prove that  
 \eqref{corolDAa} implies  \eqref{corolDAb}. Assume that \eqref{corolDAa} holds, and let $\phi\colon \fl\kappa\to\fl\lambda$ be an embedding with \DAstr-symmetric
range $S:=\phi(\fl\kappa)$. We can also assume that $\kappa\neq\lambda$ since there is nothing to prove otherwise.
Since $\phi$ is an embedding, $|\fl\kappa|\leq |\fl\lambda|$. There are two cases, depending on $\lambda$. First, if $\lambda<\aleph_0$, then 
$\kappa$ is not an odd integer by \eqref{eqpbxdzBmztpLsjflS} and, furthermore, $|\fl\kappa|\leq |\fl\lambda|=\aleph_0$ gives that $\kappa\leq \aleph_0$. Hence, \eqref{corolDAb} holds in this case.  Second, if $\lambda\geq\aleph_0$, then 
\begin{equation}
\kappa \overset{ \eqref{aligntxtDfkrZntNshG} } =|\Dirr{\fl\kappa}| = |\Dirr S| \overset{\eqref{pbxdznthHmQpfcskmpdLbnDGr}}= \lambda,
\label{eqcsRpfLnTcC}
\end{equation}
and  \eqref{corolDAb}  holds again.
The proof of  Corollary~\ref{corolDA} is complete.
\end{proof}

\begin{proof}[Proof of Theorem \ref{thmmain}]
It suffices to prove part \eqref{thmmaina}, since it implies 
 part \eqref{thmmainb}.
Let $3\leq \lambda\in\nfromo$, and assume that $\kappa=\aleph_0$ or $3\leq \kappa\in\nfromo$ is even. Then there exists a totally symmetric embedding from $\fl\kappa$ to $\fl\lambda$ by Lemma~\ref{lemmawcMsKlMm}.

Conversely, assume that $3\leq\kappa$, 
$3\leq\lambda$, and there exists a totally symmetric embedding $\phi\colon \fl\kappa\to\fl\lambda$. Let $S=\phi(\fl\kappa)$ denote the range of $\phi$; it consists of some symmetric elements of $\fl\lambda$. 
It follows that $\lambda<\aleph_0$, because otherwise there would be no symmetric element in $\fl\lambda$. Thus, we have also that $\kappa\leq \aleph_0$, because  $\kappa\leq |\fl\kappa|\leq |\fl\lambda|=\aleph_0$.
Since $S$ is invariant under the natural dual automorphism $\delta:=\delta_{\fl\lambda}$ and it is symmetric, even element-wise symmetric, \eqref{eqkndXsTWsktt} shows that $S$ is \DAstr-symmetric.
By  \eqref{eqpbxhmrKpLkBtNnS}, $\restrict\delta{\Dirr S}$ is an involution  on $\Dirr S$. 
No free generator of $\fl\lambda$ is a symmetric element of $\fl\lambda$, whereby $S$ 
is disjoint from the set of free generators of $\fl\lambda$. If  $\restrict\delta{\Dirr S}$ had a fixed point $u$, then $u$ would be a fixed point of $\delta$, so \eqref{eqtxtFpNslmnTwszpThG} would imply that $u\in \Dirr S\subseteq S$ is a free generator of $\fl\lambda$, contradicting the above-mentioned disjointness. Thus, $\restrict\delta{\Dirr S}$ has no  fixed point. 
By \eqref{eqpbxhmhzlRfgldzlvs}, the fixed-point-free involution $\restrict\delta{\Dirr S}$ acts on the $\kappa$-element set $\Dirr S$, and we conclude that $\kappa$ is not an odd integer. That is, 
$\kappa=\aleph_0$ or $\kappa\in\nfromo$ is even,  completing the proof.
\end{proof}

\begin{proof}[Proof of Corollary~\ref{corolaUtsym}]
In order to prove the ``if'' part, we can assume that $\kappa\neq\lambda$ since otherwise the identity map of $\fl\kappa$ is a required embedding. So $3\leq\lambda\in \nfromo$ and $3\leq \kappa\leq \aleph_0$. By Theorem~\ref{thmmain}\eqref{thmmaina}, $\fl\lambda$ has an element-wise symmetric sublattice $S$ such that $S\cong\fl{\aleph_0}$.
Since $\kappa\leq \aleph_0$, $\fl\lambda$ has also a sublattice $S'$ such that $\fl\kappa\cong S'$. Clearly, any isomorphism from $\fl\kappa \to S'$ is an embedding of $\fl\kappa$ into $\fl\lambda$ with symmetric range; in fact, with element-wise symmetric range. This proves the ``if'' part.

In order to prove the ``only if'' part, assume that there is an embedding
$\phi\colon\fl\kappa\to \fl\lambda$ with symmetric range $S$.
Clearly,  $|\fl\kappa|\leq |\fl\lambda|$.  
Depending on $\lambda$, there are two cases to consider. 
First, if $\lambda<\aleph_0$, then $|\fl\kappa|\leq |\fl\lambda|=\aleph_0$ yields that  $\kappa\leq \aleph_0$, as required. Second, if $\lambda\geq\aleph_0$, then \eqref{eqcsRpfLnTcC} applies and $\kappa=\lambda$, again as required.
The proof is complete. 
\end{proof}

\section{A computer program and its background}
\label{sectcompProgr}

\subsection*{Historical background}
There are various known algorithms to solve the word problem of free lattices and that of finitely presented lattices. They are  discussed in Sections 8 and 9 of Chapter XI of the monograph Freese, Je\v zek, and Nation~\cite{fjnmonograph}; see also Dean~\cite{dean-alg}, Evans~\cite{evans-alg}, McKinsey~\cite{McKinseyalg}, and Skolem~\cite{skolemalg} for the original papers. In addition to this list, there is an additional algorithm given in Cz\'edli~\cite{czg-latwp}. 
We know from \cite{fjnmonograph} that the  algorithms given by Skolem, Freese, and Herrmann run in polynomial time;  so does the one given in \cite{czg-latwp}. However, it is only Whitman's algorithm  with the modifications explained in \cite{fjnmonograph} that is fast enough for our purposes.

\subsection*{A new computer program}
The first author has developed a  Dev-Pascal 1.9.2 (Freepascal) program for the word problem of free lattices. This problem is  based on the Freese-Whitman algorithm, as it is given in  Freese, Je\v zek, and Nation~\cite{fjnmonograph}.  The program runs in Windows environment (tested only under Windows 10), and it can be downloaded from the  author's website.  The program takes its input from a text file; several  sample input files are also donwloadable. We used this program on our personal computer with IntelCore i5-4440 CPU, 3.10 GHz, and 8.00 GB RAM. 

\subsection*{Results achieved with the computer program}
First, we used the program to give alternative proofs. In particular, we used~it   
\begin{equation}
\text{to prove the (Key) Lemma~\ref{lemmakeylemma} for $n=4$.}
\label{eqtxtPrPrfvH}
\end{equation}
Also, we used the program to prove that 
\begin{equation}
\parbox{9.2cm}{for $n=3$,
the Key Lemma remains valid if we replace $m_1$, $m_2$, $m_3$ and $m_4$ by $m_5$, $m_7$, $m_8$, and $m_9$, respectively;}
\label{eqpbxPrPskGfvhT}
\end{equation}
this gives an alternative proof of Corollary~\ref{corolnhRm}. By the paragraph preceding \eqref{eqdzbnTpQrgNmGlVmh}, it would not be difficult to show that the stipulation $n=3$ can be omitted from \eqref{eqpbxPrPskGfvhT}, but this or a similar strengthening of  \eqref{eqpbxPrPskGfvhT} is not pursued.

In addition to reaffirming some results from the previous sections, we could use the program to  find an entirely new construction to prove Corollary~\ref{corolnhRm}. In order to describe it, 
we use the notation introduced in Remark~\ref{remcspRtlrS} to define a join-homomorphism $\jhom\colon \fl 3\to \Sl(3)$ and a meet-homomorphism  $\mhom\colon \fl 3\to \Sl(3)$ by the rules
\begin{equation}
\jhom(u):=\bigvee_{\sigma\in \Sym_3} \aut\sigma(u)\quad\text{and}\quad
\mhom(u):=\bigwedge_{\sigma\in \Sym_3} \aut\sigma(u).
\end{equation}
In order to ease the notation, we will write $x$, $y$, and $z$ instead of $x_0$, $x_1$, and $x_2$, respectively.
Note that the program recognizes (appropriate commands for) $\jhom$ and $\mhom$ in input files. 
Take the following ternary terms, that is, elements of $\fl 3=\Fl(x,y,z)$.
\allowdisplaybreaks{
\begin{align}
a_0&= \jhom\Bigl(\bigl(((x\vee y)\wedge z)\vee y\bigr) \wedge  \bigl(((y\vee x)\wedge z)\vee x\bigr)\Bigr), \cr
a'&=  \mhom\Bigl(\bigl(((a_0\wedge x)\vee y)\wedge z \bigr)\vee \bigl(((z\wedge x)\vee y)\wedge a_0\bigr)\Bigr),\,\,\text{ and}
\label{eqneewa} \\
b'& = \mhom\Bigl(\bigl((((x\vee y)\wedge (x\vee z))\vee a')\wedge x\bigr) \vee  \bigl(((x\wedge a_0)\vee y)\wedge z\bigr)\Bigl).  \label{eqneewb}
\end{align}}%
With $a'$ and $b'$ from \eqref{eqneewa} and \eqref{eqneewb} and their duals, $\ovl {a'}$ and $\ovl {b'}$, the program proved that 
\begin{equation}
\text{ $\set{a', \ovl{ a'}, b, \ovl {b'}}$ freely generates a sublattice of $\Fl(x,y,z)$,}
\label{eqtxthtRmspRfshlQBt}
\end{equation}
which obviously implies Corollary~\ref{corolnhRm}. Note that for each of \eqref{eqtxtPrPrfvH}, \eqref{eqpbxPrPskGfvhT}, and \eqref{eqtxthtRmspRfshlQBt}, the program ran less than a millisecond on our computer.

Finally, for a lattice term $t$, we define the \emph{total number} $\tno(t)$ of \emph{variables} of $t$ by induction as follows: $\tno(t)=1$ if $t$ is a variable and 
\[\tno(t_1\vee t_2) = \tno(t_1\wedge t_2) = \tno(t_1)+\tno(t_2).
\]
Note that, say, $x=x\wedge(x\vee y)$ in $\Fl(x,y,z)$ but $\tno(x) =1$ is distinct from $\tno(x\wedge(x\vee y))=3$. Hence, as opposed to what \eqref{eqpbxtrTtRlMnts} suggests, we do not define $\tno$ for the elements of $\Fl(x,y,z)$.
For a set $\set{t_1,\dots, t_k}$  of terms,  let 
$\tno(\set{t_1,\dots, t_k})=\tno(t_1)+\dots+\tno(t_k)
$. 
Table~\ref{eqtablE} shows how the function $\tno$  compares the \emph{terms} describing the free generating set given in \eqref{eqshhbkvlhkkvlb} for $n=3$ and those given in \eqref{eqneewa} and  \eqref{eqtxthtRmspRfshlQBt}.
Another difference between  \eqref{eqshhbkvlhkkvlb}  and 
 \eqref{eqtxthtRmspRfshlQBt} is that,  as opposed to the set $\set{a,\ovl a, b,\ovl b, x_0}$ from (the Key) Lemma~\ref{lemmakeylemma}, the program shows that $\set{ a', \ovl{ a'}, b, \ovl {b'},x}$ does \emph{not} generate freely.

\begin{table}
\begin{equation}
\lower  0.6 cm
\vbox{\tabskip=0pt\offinterlineskip
\halign{\strut #&\vrule#\tabskip=0pt plus 0pt&
#\hfill&\vrule\vrule\vrule#&
\hfill#&\vrule#&
\hfill#&\vrule#&
\hfill#&\vrule#&
\hfill#&\vrule#&
\hfill#&\vrule#&
\hfill#&\vrule#&
\hfill#&\vrule#&
\hfill#&\vrule#&
\hfill#&\vrule#&
\hfill#&\vrule#&
\hfill#&\vrule#&
#\hfill\vrule\vrule\cr
\vonal
&&\hfill    && 1st generator\kr&&\kr{}2nd generator\kr&& $\phantom{\int_{I}^{I}}\kern-9pt \tno($generating set$)$&\cr
\vonal\vonal\vonal
&&\kr{}\eqref{eqshhbkvlhkkvlb}$_{n=3}$\kr&&$\phantom{\int_{I}^{I}}\kern-9pt\tno(a)=108$\kr&&$\tno(b)=228$\kr&&$\tno(\set{a,\ovl a, b,\ovl b})=672$\kr&\cr
\vonal
&&\kr{}\eqref{eqtxthtRmspRfshlQBt}\kr&&$\phantom{\int_{I}^{I}}\kern-9pt\tno(a')=612$\kr&&\kr$\tno(b')=4008$\kr&\kr&$\tno(\set{a',\ovl{a'}, b',\ovl{b'}})=9240$\kr&\cr
\vonal}} 
\label{eqtablE}
\end{equation}
\end{table}


\begin{thebibliography}{99}


\bibitem{czg-latwp}  
 Cz\'edli, G.: On the word problem of lattices with the help of graphs.
 Periodica Mathematica Hungarica \tbf{23}, 49--58 (1991)\footnote{\red{Temporary note: available from the author's homepage.}}


\bibitem{czg-selfdual} 
 Cz\'edli, G.: 
 A selfdual embedding of the free lattice over countably many generators into the three-generated one.
 Acta Math. Hungar. \tbf{148}, 100--108 (2016)

\bibitem{deanCF}
 Dean, R. A.:
 Completely free lattices generated by partially ordered sets.
 Trans. Amer. Math. Soc. \tbf{83}, 238--249  (1956)

\bibitem{dean-alg}
  Dean, R. A.:
  Free lattices generated by partially ordered sets and preserving bounds. 
  Canad. J. Math. \tbf{16}, 136--148 (1964) 


\bibitem{dilworthLUB}
  Dilworth, R. P.: Lattices with unique complements.
  Trans. Amer. Math. Soc. \tbf{57}, 123--154 (1945) 

\bibitem{evans-alg}
  Evans, T.:
  The word problem for abstract algebras.
  London Math. Soc. \tbf{26}, 64--71 (1951)

\bibitem{fr85}
Freese, R.: 
Connected components of the covering relation in free lattices. 
Universal algebra and lattice theory (Charleston, S.C., 1984), 82-93, Lecture Notes in Math., 1149, Springer, Berlin, 1985.

\bibitem{fr87}
Freese, R.: 
Free lattice algorithms. 
Order \tbf{3}, 331--344 (1987)

\bibitem{fjnmonograph}
Freese, R., Je\v zek, J., Nation, J. B.: Free lattices. Mathematical Surveys and Monographs,
42, American Mathematical Society, Providence, RI, (1995)

\bibitem{frnat85}
Freese, R., Nation, J. B.:
Covers in free lattices. 
Trans. Amer. Math. Soc. \tbf{288}, 1--42 (1985)

\bibitem{frnat16}
Freese, R., Nation, J. B.: 
Free and finitely presented lattices. 
Lattice theory: special topics and applications. Vol. 2, 27�58, Birkh\"auser/Springer, Cham, 2016.

\bibitem{gGLTF}
 Gr\"atzer, G.: 
 Lattice Theory: Foundation. 
 Birkh\"auser, Basel (2011)


\bibitem{McKinseyalg}
 McKinsey, J. C. C.:
 The decision problem for some classes of sentences without quantifiers.
 J. Symbolic Logic \tbf{8}, 61--76 (1943)


\bibitem{nat82}
 Nation, J. B.:
 Finite sublattices of a free lattice.
 Trans. Amer. Math. Soc. \tbf{269}, 311--337 (1982)

\bibitem{nat84}
 Nation, J. B.:
 On partially ordered sets embeddable in a free lattice.  Algebra Universalis \tbf{18}, 327--333 (1984)

\bibitem{nationbook}
 Nation, J. B.: Notes on Lattice Theory. 
 \url{http://math.hawaii.edu/~jb/math618/LTNotes.pdf}


\bibitem{rivalwille}
 Rival, I,  Wille, R.: 
 Lattices freely generated by partially ordered sets: which can be ``drawn''?.
 J. Reine Angew. Math. \tbf{310}, 56--80 (1979)

\bibitem{skolemalg}
  Skolem, T.:
  Selected works in logic. Edited by Jens Erik Fenstad Universitetsforlaget, Oslo 1970 732 pp.

\bibitem{tschantz}
 Tschantz, S., T.:
 Infinite intervals in free lattices.
 Order \tbf{6}, 367--388 (1990)  

\bibitem{whitman} 
 Whitman, P.: 
 Free lattices. 
 Annals of Math. \tbf{42}, 325--330 (1941)

\end{thebibliography}
\end{document}